\documentclass[11pt]{article}%
\usepackage{amssymb,amsmath,amsfonts,amsthm,array,bm,bbm,color}
\usepackage{amsmath}
\usepackage{amsfonts}
\usepackage{amssymb}
\usepackage{graphicx}%
\usepackage{subcaption}
\usepackage{dsfont}
\providecommand{\U}[1]{\protect\rule{.1in}{.1in}}
\numberwithin{equation}{section}
\newtheorem{theorem}{Theorem}[section]
\newtheorem{lemma}[theorem]{Lemma}
\newtheorem{corollary}[theorem]{Corollary}

\newtheorem{remark}[theorem]{Remark}

\def\<{\langle}
\def\>{\rangle}
\def\d{{\rm d}}

\def\E{\mathbb{E}}

\def\R{\mathbb{R}}
\def\RR{\mathbb{R}}

\def\eps{\varepsilon}

\newcommand\abs[1]{\left| #1 \right|}
\newcommand\norm[1]{\left\| #1 \right\|}

\begin{document}

\title{Numerical computation of probabilities for nonlinear SDEs in high dimension using Kolmogorov equation}

\author{Franco Flandoli\footnote{Email: franco.flandoli@sns.it. Scuola Normale Superiore, Piazza dei Cavalieri, 7, 56126 Pisa, Italy.}
\quad Dejun Luo\footnote{Email: luodj@amss.ac.cn. Key Laboratory of RCSDS, Academy of Mathematics and Systems Science, Chinese Academy of Sciences, Beijing 100190, China, and School of Mathematical Sciences, University of the Chinese Academy of Sciences, Beijing 100049, China.}
\quad Cristiano Ricci\footnote{Email: cristiano.ricci@sns.it. Scuola Normale Superiore, Piazza dei Cavalieri, 7, 56126 Pisa, Italy.}}
\maketitle

\begin{abstract}
Stochastic Differential Equations (SDEs)\ in high dimension, having the
structure of finite dimensional approximation of Stochastic Partial
Differential Equations (SPDEs), are considered. The aim is to compute
numerically expected values and probabilities associated to their solutions,
by solving the associated Kolmogorov equations, with a partial use of Monte
Carlo strategy - precisely, using Monte Carlo only for the linear part of the
SDE. The basic idea was presented in \cite{FLR}, but here we strongly improve
the numerical results by means of a \textit{shift} of the auxiliary Gaussian
process. For relatively simple nonlinearities, we have good results in
dimension of the order of 100.
\end{abstract}


\section{Introduction}

In a finite dimensional space $\mathbb{R}^{d}$, having in mind the case when
$d$ is high, we consider an SDE of the form
\begin{equation}\label{SDE}
\left\{ \aligned
\mathrm{d}X_{t}  &  =\left(  AX_{t}+B_{0}\left(
t,X_{t}\right)  \right)  \mathrm{d}t+\sqrt{Q}\, \mathrm{d}W_{t}\qquad\text{for
}t\geq t_{0},\\
X_{t_{0}}  &  =x_{0}. \endaligned \right.
\end{equation}
Here $A$ is a $d\times d$ matrix, assumed to be self-adjoint and strictly
negative definite, $Q$ is a self-adjoint and strictly positive $d\times d$
matrix (the identity, up to a constant, in all our examples), $\sqrt{Q}$ is
the square root of $Q$, $W_{t}$ is a $d$-dimensional Brownian motion defined
on a probability space $( \Omega,\mathcal{F},\mathbb{P}) $ with expectation
$\mathbb{E}$, $x_{0}\in\mathbb{R}^{d}$ and $B_{0}:\left[  0,T\right]
\times\mathbb{R}^{d}\rightarrow\mathbb{R}^{d}$ is a locally Lipschitz
continuous (uniformly in $t$) vector field with additional properties which
guarantee global existence (uniqueness coming from locally Lipschitz
assumption) of the solution denoted hereafter by $X_{t_{0},t}^{x_{0}}$. This
kind of structure, in particular the presence of the operator $A$, is inspired
by finite dimensional approximation of SPDEs and indeed, in Section \ref{subsec:spectral models}, we consider examples coming from the finite dimensional approximation of certain SPDEs. The presence of the
negative definite operator $A$ is crucial for the algorithm described below
and it is a main distinctive feature compared to \cite{Beck, BBGJJ, ChenMajda, EHJ, E
Hutz Jent Kru 1}. In \cite{FLR} it is shown that thanks to $A$, the problem
can be stated in Hilbert spaces and the convergence estimates are then
dimension independent.

The aim is to compute numerically expected values of the form $\mathbb{E}%
\big[  \phi\big(  X_{t_{0},T}^{x_{0}}\big)  \big]  $ (and probabilities
associated to $X_{t_{0},T}^{x_{0}}$, by taking $\phi$ equal to indicator
functions), by solving the corresponding Kolmogorov equation ($\left\langle
\cdot,\cdot\right\rangle $ denotes the scalar product in $\mathbb{R}^{d}$ and
$D$ the spatial derivative)
\begin{equation}\label{backward Kolmogorov}
\left\{ \aligned
\partial_{t}U+\frac{1}{2}\mathrm{Tr}%
\left(  QD^{2}U\right)  +\left\langle Ax+B_{0}\left(  t,x\right)  ,D
U\right\rangle  &  =0\qquad\text{for }t\leq T,\\
U\left(  T,x\right)   &  =\phi\left(  x\right)  \endaligned \right.
\end{equation}
due to the fact that
\[
\mathbb{E}\big[  \phi\big(  X_{t_{0},T}^{x_{0}}\big)  \big]  =U\left(  t_{0},
x_{0}\right)  .
\]
Remember that a direct solution (without Monte Carlo)\ of the Kolmogorov
equation is a typical example of \textit{curse of dimensionality};\ the key
progress of the present approach is to allow this numerical computation with
only a partial use of Monte Carlo - precisely, using Monte Carlo only for the
linear part of the SDE. The basic idea was presented in \cite{FLR}, but here
we strongly improve the numerical results by means of an auxiliary
\textit{shift} of the Gaussian process. For relatively simple nonlinearities,
we have good results in dimension of the order of 100; and much better than
those obtained without shift. Let us describe this idea.

Consider the stochastic equation (\ref{SDE}). Let $f\in C\left(  \left[
0,T\right]  ;\mathbb{R}^{d}\right)  $ be given; we rewrite the equation as%
\[
\mathrm{d}X_{t}=\left(  AX_{t}+f\left(  t\right)  +B\left(  t,X_{t}\right)
\right)  \mathrm{d}t+\sqrt{Q}\, \mathrm{d}W_{t},
\]
where%
\[
B\left(  t,x\right)  =B_{0}\left(  t,x\right)  -f\left(  t\right)  .
\]
Consider the auxiliary Gaussian process $Z_{t_{0},t}^{x_{0}}$ solution of%
\begin{equation}\label{eq:Zshifted}
\left\{ \aligned
\mathrm{d}Z_{t}  &  =\left(  AZ_{t}+f\left(  t\right)  \right)  \mathrm{d}%
t+\sqrt{Q}\mathrm{d}W_{t}\qquad\text{for }t\geq t_{0},\\
Z_{t_{0}}  &  =x_{0}.
\endaligned \right.
\end{equation}
Let $P_{t_{0},t}$ and $S_{t_{0},t}$ be the Kolmogorov evolution operators
associated to processes $X_{t_{0},t}^{x_{0}}$ and $Z_{t_{0},t}^{x_{0}}$
respectively:%
\begin{align*}
\left(  P_{t_{0},t}\phi\right)  \left(  x_{0}\right)   &  =\mathbb{E}\left[
\phi\left(  X_{t_{0},t}^{x_{0}}\right)  \right] ,\\
\left(  S_{t_{0},t}\phi\right)  \left(  x_{0}\right)   &  =\mathbb{E}\left[
\phi\left(  Z_{t_{0},t}^{x_{0}}\right)  \right]  .
\end{align*}
We are interested in their relationship because then we compute statistical
quantities associated to $X_{t_{0},t}^{x_{0}}$ using a Gaussian process. The
following lemma is essentially well known, but we provide a proof in the
Appendix for the readers' convenience.

\begin{lemma}
\label{thm-1} Assume that the drift $B\in C_{b}^{0,1} \big([0,T]\times
\mathbb{R}^{d}, \mathbb{R}^{d} \big)$; then
\begin{equation}
\label{thm-1.1}\left(  P_{t_{0},t}\phi\right)  \left(  x\right)  =\left(
S_{t_{0},t}\phi\right)  \left(  x\right)  +\int_{t_{0}}^{t} \left(
S_{t_{0},s} \left\langle B\left(  s,\cdot\right)  ,D P_{s,t}\phi\right\rangle
\right)  \left(  x\right)  \mathrm{d}s.
\end{equation}

\end{lemma}

The presence of the shift $f\left(  t\right)  $ is the characteristic feature
of this work with respect to \cite{FLR} and it is introduced to improve
(enormously) the precision of the scheme. Although different choices are
possible, in this paper we always specify $f\left(  t\right)  $ as%
\begin{equation}\label{eq:fShift}
f\left(  t\right)  =B_{0}\left(  t,x\left(  t\right)  \right) ,
\end{equation}
where $x\left(  t\right)  $ is the unique solution of the deterministic
equation%
\begin{equation}\label{eq:odemean}
\left\{ \aligned
\frac{\mathrm{d} x\left(  t\right)  }{\mathrm{d} t}  &  =Ax\left(  t\right)
+B_{0}\left(  t,x\left(  t\right)  \right)  ,\\
x\left(  0\right)   &  =x_{0}.
\endaligned \right.
\end{equation}
Notice that equation \eqref{eq:odemean} is the deterministic counterpart of \eqref{SDE}.
The rationale behind this choice of $f(t)$ is the following: we imagine that the solution to the deterministic equation \eqref{eq:odemean} are somewhat close to the mean value of the stochastic process \eqref{SDE}. Of course this is not strictly true in some examples, depending on the shape of the nonlinearity $B_{0}$ or the matrix $Q$, but in most cases it should be a reasonable approximation. By choosing the shift $f$ as in \eqref{eq:fShift} we have that
the mean value of the stochastic process $Z_{t}$, defined in \eqref{eq:Zshifted}, coincides with the solution of the deterministic equation \eqref{eq:odemean}. Hence by this choice the Gaussian process $Z_{t}$ and the nonlinear one $X_{t}$ have expected values that are close. This will pose a significant advantage when trying to exploit the strategy described below to compute a numerical approximation of Kolmogorov equation, as described in Section \ref{sec:visualization}.

We use identity (\ref{thm-1.1}) iteratively, to approximate $\left(
P_{t_{0},t}\phi\right)  \left(  x\right)  $. However, storing the iterative
information directly (namely the function $D P_{s,t}^{n}\phi$
corresponding to the $n$ iteration) is too costly. Following \cite{FLR} we
rewrite each iteration in terms of the (shifted) Gaussian process $Z_{t_{0}%
,t}^{x_{0}}$ above. The iteration is described theoretically in the next
section and numerically in the subsequent one.

\section{The iteration scheme with the shift}\label{sec:iterationscheme}

Given $\left(  s,x\right)  \in\left[  0,T\right]  \times\mathbb{R}^{d}$, let
$(Z_{s,t}^{x})_{t\in\left[  s,T\right]  }$ be the solution to
\begin{equation}\label{eq:gaussianShift}
\mathrm{d}Z_{s,t}^{x}=(AZ_{s,t}^{x}+f(t))\,\mathrm{d}t+\sqrt{Q}\,\mathrm{d}%
W_{t},\quad t\in\left[  s,T\right]  ,\,Z_{s,s}^{x}=x.
\end{equation}
It has the expression
\begin{equation}
Z_{s,t}^{x}=e^{(t-s)A}x+F_{s,t}+W_{A}(s,t), \label{modified-Gaussian}%
\end{equation}
where $e^{tA}$ is the matrix exponential function associated to $A$,
$F_{s,t}=\int_{s}^{t}e^{(t-r)A}f(r)\,\mathrm{d}r$ and $W_{A}(s,t)$ is the
stochastic convolution
\[
W_{A}(s,t)=\int_{s}^{t}e^{(t-r)A}\sqrt{Q}\,\mathrm{d}W_{r},
\]
which has the covariance matrix $Q(s,t)=Q_{t-s}$, the latter being defined as
\[
Q_{t}=\int_{0}^{t}e^{sA}Qe^{sA^{\ast}}\,\mathrm{d}s.
\]
Under our assumptions, the matrix $Q_{t}$ is invertible for all $t\geq0$.

The formula (\ref{thm-1.1}) suggests us to consider the iteration scheme:
\[
u_{s,t}^{0}(x)=\left(  S_{s,t}\phi\right)  (x)=\mathbb{E}\left[  \phi
(Z_{s,t}^{x})\right]
\]
and, for $n\geq0$,
\begin{equation} \label{iteration}
u_{s,t}^{n+1}(x)=S_{s,t}\phi(x)+\int_{s}^{t}\left(  S_{s,r}\left\langle
B(r,\cdot),Du_{r,t}^{n}\right\rangle \right)  (x)\,\mathrm{d} r, \quad t\in (s,T].%
\end{equation}
We introduce the notation $v_{s,t}^{0}(x)=u_{s,t}^{0}(x)=\left(  S_{s,t}%
\phi\right)  (x)$ and for $n\geq1$,
\[
v_{s,t}^{n}(x)=u_{s,t}^{n}(x)-u_{s,t}^{n-1}(x),\quad x\in\mathbb{R}%
^{d},\, t\in (s,T].
\]
Then the new functions satisfy the iteration scheme below:%
$$\left\{ \aligned
v_{s,t}^{0}(x)  &  =\left(  S_{s,t}\phi\right)  (x),\\
v_{s,t}^{n+1}\left(  x\right)   &  =\int_{s}^{t}\left(  S_{s,r}k_{r,t}%
^{n}\right)  (x)\, \mathrm{d}r,\quad\text{for }n\geq0\text{, where}\\
k_{r,t}^{n}(y)  &  =\left\langle B(r,y),Dv_{r,t}^{n}(y)\right\rangle .
\endaligned \right. $$

We need the following lemma. Denote by $\mathcal{B}(\mathbb{R}^{d})$ the
family of bounded measurable functions on $\mathbb{R}^{d}$; notice that
$\phi\in\mathcal{B}(\mathbb{R}^{d})$ is not differentiated in the formula
below, an essential point also for the purpose of computing probabilities
associate to $X_{s,t}^{x}$. We introduce the notation $\Lambda(t)=
Q_{t}^{-1/2} e^{tA},\, t\geq0$.

\begin{lemma}
\label{lem-derivative} For any $h\in\mathbb{R}^{d}$, $\phi\in\mathcal{B}%
(\mathbb{R}^{d})$ and $0<s<t$,
\begin{equation}
\left\langle h,D(S_{s,t}\phi)(x)\right\rangle =\mathbb{E}\left[  \phi
(Z_{s,t}^{x})\left\langle \Lambda(t-s)h,Q_{t-s}^{-1/2}\big( Z_{s,t}%
^{x}-e^{(t-s)A}x-F_{s,t}\big) \right\rangle \right]  ,\quad x\in\mathbb{R}%
^{d}. \label{lem-derivative.1}%
\end{equation}
Note that, by the formula (\ref{modified-Gaussian}), $Q_{t-s}^{-1/2}\left(
Z_{s,t}^{x}-e^{(t-s)A}x-F_{s,t}\right)  =Q_{t-s}^{-1/2}W_{A}(s,t)$ is a
standard Gaussian random variable in $\mathbb{R}^{d}$.
\end{lemma}

The proof is given in the Appendix. Now we can prove:

\begin{corollary}
\label{cor-first-itera} One has, for $0\leq s<t\leq T$,
\begin{equation}
v_{s,t}^{1}(x)=\int_{s}^{t}\mathbb{E}\left[  \phi(Z_{s,t}^{x})\left\langle
\Lambda(t-r)B(r,Z_{s,r}^{x}),Q_{t-r}^{-1/2}\big( Z_{s,t}^{x}-e^{(t-r)A}%
Z_{s,r}^{x}-F_{r,t}\big) \right\rangle \right]  \,\mathrm{d}r.
\label{cor-first-itera.1}%
\end{equation}

\end{corollary}

\begin{proof}
By the definition of $v_{r,t}^{0}$ and Lemma \ref{lem-derivative}, we have
\[
k_{r,t}^{0}(y)=\left\langle B(r,y),D(S_{r,t}\phi)(y)\right\rangle
=\mathbb{E}\left[  \phi(Z_{r,t}^{y})\left\langle \Lambda(t-r)B(r,y),Q_{t-r}%
^{-1/2}\big(  Z_{r,t}^{y}-e^{(t-r)A}y-F_{r,t}\big)  \right\rangle \right]
.
\]
Therefore, for $s<r<t$,%
\begin{align*}
k_{r,t}^{0}(Z_{s,r}^{x})  & =\mathbb{E}\left[  \phi(Z_{r,t}^{y})\left\langle
\Lambda(t-r)B(r,y),Q_{t-r}^{-1/2}\left(  Z_{r,t}^{y}-e^{(t-r)A}y-F_{r,t}%
\right)  \right\rangle \right]  _{y=Z_{s,r}^{x}}\\
& =\mathbb{E}\left[  \phi(Z_{s,t}^{x})\left\langle \Lambda(t-r)B(r,Z_{s,r}%
^{x}),Q_{t-r}^{-1/2}\big(  Z_{s,t}^{x}-e^{(t-r)A}Z_{s,r}^{x}-F_{r,t}\big)
\right\rangle \Big|Z_{s,r}^{x}\right]  \\
& =\mathbb{E}\left[  \phi(Z_{s,t}^{x})\left\langle \Lambda(t-r)B(r,Z_{s,r}%
^{x}),Q_{t-r}^{-1/2}\big(  Z_{s,t}^{x}-e^{(t-r)A}Z_{s,r}^{x}-F_{r,t}\big)
\right\rangle \Big|\mathcal{F}_{s,r}\right]
\end{align*}
by Markov property. Here $\mathcal{F}_{s,r}$ is the $\sigma$-algebra generated by the random variables $W_{t_2}- W_{t_1}$ with $s\leq t_1< t_2\leq r$. As a result,
\begin{align*}
v_{s,t}^{1}(x)  & =\int_{s}^{t}(S_{s,r}k_{r,t}^{0})(x)\,\mathrm{d}r=\int%
_{s}^{t}\mathbb{E}\left[  k_{r,t}^{0}(Z_{s,r}^{x})\right]  \,\mathrm{d}r\\
& =\int_{s}^{t}\mathbb{E}\left[  \phi(Z_{s,t}^{x})\left\langle \Lambda
(t-r)B(r,Z_{s,r}^{x}),Q_{t-r}^{-1/2}\big(  Z_{s,t}^{x}-e^{(t-r)A}Z_{s,r}%
^{x}-F_{r,t}\big)  \right\rangle \right]  \,\mathrm{d}r
\end{align*}
which is (\ref{cor-first-itera.1}).
\end{proof}

The general formula is given by the following theorem. The proof is identical
to the one of \cite[Corollary 2.10]{FLR}, under the shift modification as above.

\begin{theorem}
For every $n\geq1$, setting $s_{n+1}=t$, then for all $0\leq s<t$, we have
\[
\aligned
& \, v^{n}_{s,t}(x) =  \int_{s}^{t} \mathrm{d} s_{n}\int_{s}^{s_{n}}
\mathrm{d} s_{n-1} \cdots\int_{s}^{s_{2}} \mathrm{d} s_{1} \\ & \,
\mathbb{E}\Bigg[ \phi(Z^{x}_{s,t}) \prod_{i=1}^{n} \Big\<\Lambda(s_{i+1}%
-s_{i}) B(s_{i}, Z^{x}_{s,s_{i}}), Q_{s_{i+1}-s_{i}}^{-1/2} \big(Z^{x}%
_{s,s_{i+1}} - e^{(s_{i+1}-s_{i})A} Z^{x}_{s,s_{i}} - F_{s_{i},s_{i+1}}
\big) \Big> \Bigg].  \endaligned
\]
\begin{remark}\label{remark:iterativeIn}
As done in \cite[Section 3]{FLR} it is more convenient for numerical purposes to rewrite $v^{n}_{s,t}(x)$ in a different manner
\[
v^{n}_{s,t}(x) = \E\big[\phi( Z_{s,t}^x) I^{n}_{s,t}(x) \big],\quad I^{0}_{s,t}(x) \equiv 1,
\]
where, for $n\geq 0$,
\begin{equation}\label{eq:numericalI^Nrecursive}
I^{n+1}_{s,t}(x) = \int_{s}^{t}\left\langle
\Lambda(t-r)B(r,Z_{s,r}^{x}),Q_{t-r}^{-1/2}\big( Z_{s,t}^{x}-e^{(t-r)A}%
Z_{s,r}^{x}-F_{r,t}\big) \right\rangle I^{n}_{s,r}(x)\,\d r.
\end{equation}
This will allow us to compute subsequent terms of $v^{n}_{s,t}(x)$ in an iterative manner. This interpretation is also related to the work \cite{FLRGirsanov} in which the relation between the iteration scheme for Kolmogorov equation and Girsanov transformation is investigated.
\end{remark}
\end{theorem}

\section{Numerical results}\label{sec:numericalResults}
In this section we present some numerical results obtained from the iteration scheme introduced in Section \ref{sec:iterationscheme}. As stated in introduction we have in mind the finite dimensional approximation of nonlinear SPDEs. The case we aim to solve is that of nonlinear SPDEs with additive noise
\[
\d u = (\Delta u + B(u))\,\d t + \sigma \sqrt{Q}\,\d W_t,\quad u|_{t=0} = u_{0}
\]
on $\mathbb{T}^{N} = \RR^{N}/\mathbb{Z}^{N}$. Hence in our numerical examples we will assume the operator $A$ to be a suitable discretization of the Laplacian operator $\Delta$ on $\mathbb{T}^{N}$. Moreover we will also consider the case where the dimension of the underlying space $N$ is one, and hence we will take $Q$ to be the identity operator. The parameter $\sigma$ in front of the noise is present for the sake of generality, in order to have different amplitudes for the noise. In what follows we will test the iteration scheme in Section \ref{sec:iterationscheme} by Fourier discretization. We will first describe the general strategy, and present the numerical results in Section \ref{subsec:spectral models}. In the sequel we set the starting time $t_0$ (or $s$) to be 0 and write the processes as $X^{x_0}_t$ and $Z^{x_0}_t$ etc.

We will compare the results obtained by iterations with an approximated reference solution, obtained by the standard Monte-Carlo approach. Namely, we take as an approximation of $u(T,x_{0})$ the function $u^{ref}(T,x_{0})$ computed by averaging $N_{s}$ independent samples of the Euler-Maruyama time-discretization of the process $X^{x_{0}}_{T}$:
\[
u(T,x_{0}) \approx u_{ref}^{N_{s},\Delta_{e}t}(T,x_{0}) := \frac{1}{N_{s}} \sum_{i=1}^{N_{s}} \phi\Big(X^{x_{0},i, \Delta_{e}t}_{T/\Delta_{e}t} \Big),
\]
where $X^{x_{0},i,\Delta_{e}t}_{j}$ is defined  for all $i$ and for $j=1,\dots, T/\Delta_{e}t$ (here we are implicitly assuming $T/\Delta_{e}t$ to be an integer) by
\begin{equation}\label{eq:Xeuler}
\left\{ \aligned
X^{x_{0},i,\Delta_{e}t}_{j} &= X^{x_{0},i,\Delta_{e}t}_{j-1} + \Delta_{e}t \Big( A X^{x_{0},i,\Delta_{e}t}_{j-1} + B_{0}\big((j-1)\Delta_{e}t, X^{x_{0},i, \Delta_{e}t}_{j-1} \big) \Big) \\
   &\quad +  \sqrt{\Delta_{e}t}\, \big(W^{i}_{j\Delta_{e}t} - W^{i}_{(j-1)\Delta_{e}t} \big),\\
X^{x_{0},i,\Delta_{e}t}_{0} &= x_{0},
\endaligned \right.
\end{equation}
and where $\{W^{i}_{t} \}_{i=1}^{N_s}$ are independent $d$-dimensional Brownian Motions. The reason for the notation $\Delta_{e}t$ will become clear below, when we describe the mixed-time-step strategy.

Next we describe the numerical scheme to implement the iteration procedure described in Section \ref{sec:iterationscheme}. We need to discretize the ordinary differential equation \eqref{eq:odemean} and the Gaussian process \eqref{modified-Gaussian}. We start by the discretization of \eqref{eq:odemean}: define
\begin{equation}\label{eq:odeEuler}
\begin{cases}
y^{\Delta_{e}t}_{j} = y^{\Delta_{e}t}_{j-1} + \Delta_{e}t \big(A y^{\Delta_{e}t}_{j-1} + B_{0}\big((j-1)\Delta_{e}t, y^{\Delta_{e}t}_{j-1} \big) \big) \quad j =1,\dots, T/\Delta_{e}t,\\
y^{\Delta_{e}t}_{0} = x_{0}.
\end{cases}
\end{equation}
This discretization corresponds to the classical explicit Euler scheme for ODE. As pointed out in the introduction, the system \eqref{eq:odeEuler} is the deterministic counterpart of the original SPDE problem, discretized in \eqref{eq:Xeuler}. We can now define for $j=0,\dots,T/\Delta_{e}t$
\begin{equation}\label{eq:fShiftDiscrete}
f^{\Delta_{e}t}_{j} = B_0\big(j\Delta_{e}t,y^{\Delta_{e}t}_{j} \big)
\end{equation}
that is the discrete counterpart of \eqref{eq:fShift}.

In order to implement the iteration scheme with shift presented in Section \ref{sec:iterationscheme}, we have to introduce also the discretization of the Gaussian process \eqref{modified-Gaussian}. To approximate the Gaussian process $Z_{t}$ starting from zero (i.e., $x_0=0$ and $f(t)\equiv 0$ in \eqref{eq:Zshifted}), we define for  $i = 1,\dots, N_{s}$ and for $j=1,\dots, T/\Delta_{e}t $
\begin{equation}\label{eq:Zeuler}
\begin{cases}
Z^{i,\Delta_{e}t}_{j} = Z^{i,\Delta_{e}t}_{j-1} + \Delta_{e}t A Z^{i,\Delta_{e}t}_{j-1} + \sqrt{\Delta_{e}t}\,\big(W^{i}_{j\Delta_{e}t} - W^{i}_{(j-1)\Delta_{e}t} \big),\\
Z^{i,\Delta_{e}t}_{0} = 0.
\end{cases}
\end{equation}
However, as previously done in \cite{FLR}, we decide to adopt a mixed-time-step strategy to compute approximations of $u^{n}(t,x)= u^n_{0,t}(x)$. Hence we take $\Delta_{q}t \gg \Delta_{e}t$ and perform all the numerical approximations needed to compute integrals of $v^{n}$  by using the time step $\Delta_{q}t$. Define the sampling of  $Z^{i,\Delta_{e}t}_{j}$ only at time steps that are multiple of $\Delta_{q}t$ (under the assumption that $\Delta_{q}t/\Delta_{e}t$ is an integer) by
\[
 Z^{i,\Delta_{q}t}_{j'} =  Z^{i,\Delta_{e}t}_{j' \frac{\Delta_{q}t}{\Delta_{e}t} }, \quad j' = 0,\dots, T/\Delta_{q}t .
\]
We adopt the same sampling mechanism for the shift in \eqref{eq:fShiftDiscrete}:
\[
 f^{\Delta_{q}t}_{j'} =  f^{\Delta_{e}t}_{j' \frac{\Delta_{q}t}{\Delta_{e}t} }, \quad j' = 0,\dots, T/\Delta_{q}t.
\]
Moreover, we need to introduce the discretization of the shifted process in \eqref{modified-Gaussian}. This is obtained by numerically integrating via the rectangle rule the function $ f^{\Delta_{q}t}_{j}$. For $j,k = 0,\dots, T/\Delta_{q}t $, let $F^{\Delta_{q}t}_{j,k}$ be the discretization of the function $F_{s,t}$ in \eqref{modified-Gaussian}:
\begin{equation}\label{eq:FstDiscrete}
\begin{cases}
F^{\Delta_{q}t}_{j,k} = \Delta_{q}t \sum_{l=j}^{k-1}e^{\Delta_{q}t(k-l)A}f^{\Delta_{q}t}_{l},\quad &\text{ if } k > j, \\
F^{\Delta_{q}t}_{j,k}  = 0, &\text{ otherwise },
\end{cases}
\end{equation}
where $e^{\Delta_{q}t(k-l)A}$ for $A$ a finite dimensional matrix is the classic matrix exponential. Now we can present the discretization of expression \eqref{modified-Gaussian}: for $j=0,\dots, T/\Delta_{q}t$ and $i = 1,\dots, N_{s}$,
\begin{equation}
\label{eq:Znumerics}
Z^{x_{0},i,\Delta_{q}t}_{j} = e^{j\Delta_{q}tA}x_{0} + F^{\Delta_{q}t}_{0,j} + \sigma Z^{i,\Delta_{q}t}_{j}.
\end{equation}

Finally, we can illustrate the discretized iteration procedure introduced in Remark \ref{remark:iterativeIn}. Let
\[
I^{0,x_{0},i,\Delta_{q}t,}_{j} \equiv 1 \quad \forall\, i=1,\dots,N_{s} \text{ and } j=0,\dots, T/\Delta_{q}t,
\]
and for $i = 1,\dots, N_{s}$ we define $I^{n+1,x_{0},i,\Delta_{q}t,}_{j} = 0$ if $j = 0$ and
 {\small \begin{align}\label{eq:InDiscrete}
&I^{n+1,x_{0},i,\Delta_{q}t,}_{j} \\
&= \Delta_q t \sum_{l=1}^{j} \left\langle \Lambda^{\Delta_{q}t}_{j-l+1}  B\left(l\Delta_q t,  Z^{x_{0},i,\Delta_q t}_{l}\right)
,Q_{j-l+1}^{\Delta_{q}t,-1/2}\left(  Z^{x_{0},i,\Delta_q t}_{j}-e^{(j-l+1)\Delta_{q}tA}Z^{x_{0},i,\Delta_q t}_{l} - F^{\Delta_{q}t}_{l,j}\right)   \right\rangle I^{n,x_{0},i,\Delta_q t}_{l} \nonumber
\end{align} }
for $j = 1,\dots, T/\Delta_{q}t$. In the above formula we also used the finite dimensional matrices $\Lambda^{\Delta_{q}t}_{j}$ and $Q^{\Delta_{q}t,-1/2}_{j}$, $j = 0,\dots, T/\Delta_{q}t$, which are obtained by discretizing respectively the operators $\Lambda(t)$ and $Q^{-1/2}_{t}$ in the finite dimensional case. We recall that, as stated in the introduction, the function $B$ is obtained from $B_{0}$ by subtracting the corresponding shift $f$. Hence, in the discrete formula \eqref{eq:InDiscrete} we abused a little the notation by writing $B\left(j\Delta_q t,  Z^{x_{0},i,\Delta_q t}_{j}\right)$ in place of
\[
B_{0}\left(j\Delta_q t,  Z^{x_{0},i,\Delta_q t}_{j}\right) - f^{\Delta_{q}t}_{j},\quad \forall j = 0,\dots, T/\Delta_{q}t.
\]
In the previous formula we are integrating in time each sample path coming from $Z^{x_0}_{t}$. Hence we define the approximating functions as a Monte Carlo average for each $j =1,\dots, T/\Delta_q t$:
\[
v^{n,x_{0},N_{s},\Delta_q t}_{j} = \frac{1}{N_{s}} \sum_{i=1}^{N_{s}} \phi\Big( Z^{x_{0},i,\Delta_q t}_{j} \Big) I^{n,x_{0},i,\Delta_q t}_{j},\quad u^{n,x_{0},N_{s},\Delta_q t}_{j}  =  u^{n-1,x_{0},N_{s},\Delta_q t}_{j} + v^{n,x_{0},N_{s},\Delta_q t}_{j}.
\]
Since the numerical scheme is iterative, as done in \cite{FLR}, we adopt a consecutive-iterations stopping condition. At every step of computation we measure the difference between consecutive iterations and stop when this difference is below a certain threshold $tol$. At each iteration we measure
\[
err(n) := \sup_{j= 1,\dots, T/\Delta_q t}\abs{v^{n,x_{0},N_{s},\Delta_q t}_{j}}
\]
and stop the procedure if $err(n) < tol$. In each example we will also present a plot of $err(n)$ (\emph{iterative error}), as well as a comparison with the reference solution at each step of the iteration (\emph{absolute error}) in logarithmic scale.

We point out the computational advantage of this iterative strategy compared to classical Monte Carlo. In order to compute iterations we only need to produce samples of the process $Z_{t}^{x_{0}}$. However, thanks to expression \eqref{modified-Gaussian} (or \eqref{eq:Znumerics} for its discrete counterpart) we know that it is possible to produce samples of $Z_{t}^{x_{0}}$  from independent realization of the stochastic convolution $W_A(t)= W_{A}(0,t)$. This is of great computational advantage, since it allows us to compute a fixed set of independent trajectories from $W_{A}(t)$ and use them for all future computations. In particular, we highlight that it is possible to use the same samples from $W_{A}(t)$ in every case where the linear part $A$ is kept fixed. This means that, if one wants to perform multiple simulations relative to $X^{x_{0}}_{t}$ by changing the initial condition $x_{0}$ (due to uncertainty on initial data, sensibility analysis, etc.), the nonlinear term $B(t,x)$ or even the magnitude of the noise $\sigma$, the same independent samples from $W_{A}(t)$ can be used. Hence a large amount of computational time can be saved by only computing deterministic time integration for each of the trajectories of $Z_{t}^{x_{0}}$, without having to repeat the Euler-Maruyama scheme \eqref{eq:Zeuler} each time.

We can now proceed with presenting the numerical results obtained by specializing the strategy introduced above by means of spectral methods. In each case we will detail the particular choice of numerical parameters and that of the model. Here we only specify those which are common to all the experiments performed below.
We always take $N_{s} = 10^{4}$ as the number of samples and $\Delta_{q}t =10^{-2}$ for the discretization of the algorithm. By this choice we expect the final error to be proportional to $err = 10^{-2}$ that is the case of the standard Monte Carlo approach, together with explicit Euler-Maruyama discretization. To compute reference solutions, as well as the trajectories of the Gaussian process \eqref{eq:Zeuler} we will use a value of $\Delta_{e}t$ much smaller than $\Delta_q t$, and a higher number of samples to take averages.

\begin{figure}[t]
\begin{subfigure}{0.49\textwidth}
\includegraphics[width=\textwidth]{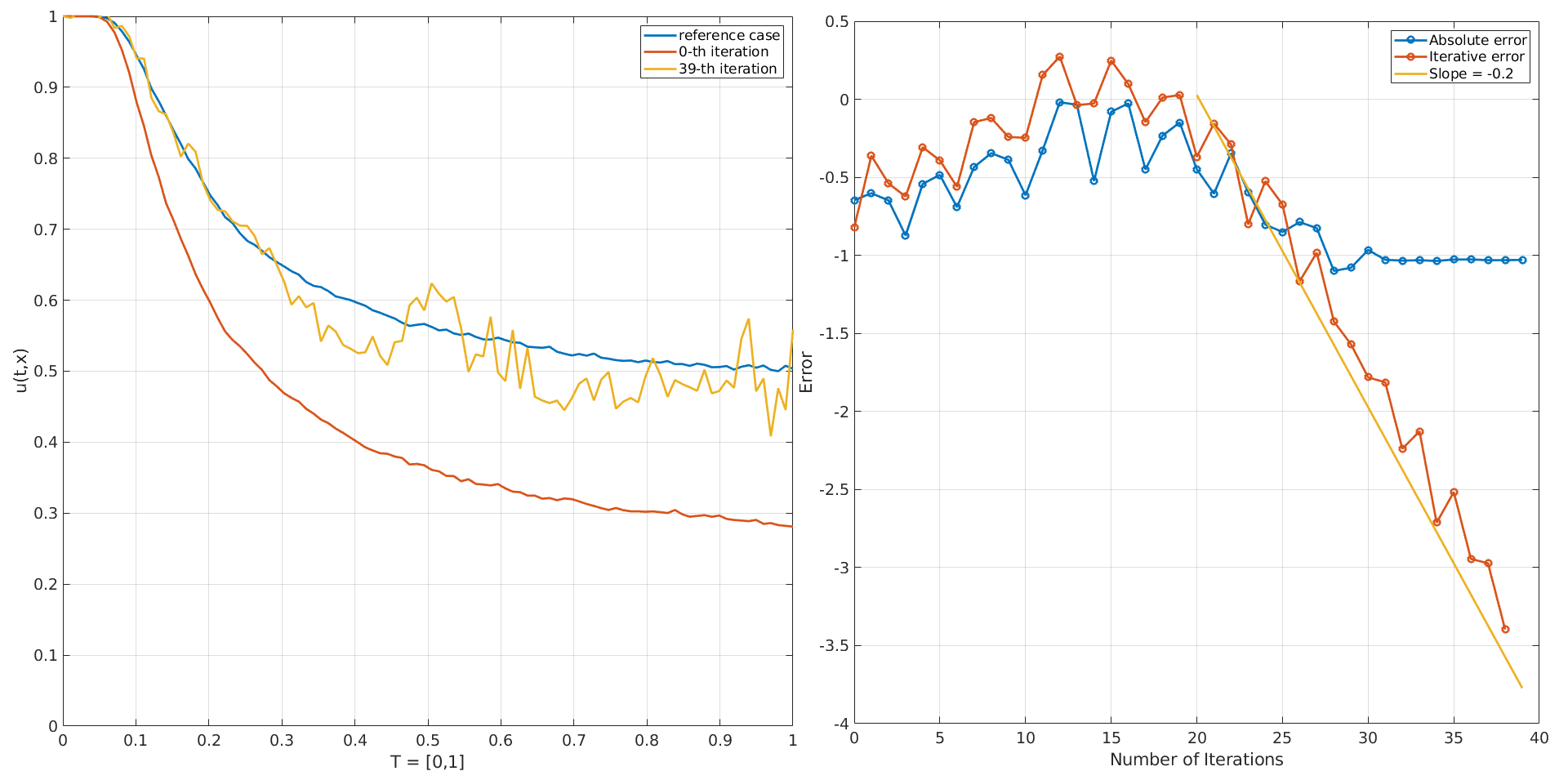}
\label{fig:cubicBoundedNoShiftd10}
\end{subfigure}
\begin{subfigure}{0.49\textwidth}
\includegraphics[width=\textwidth]{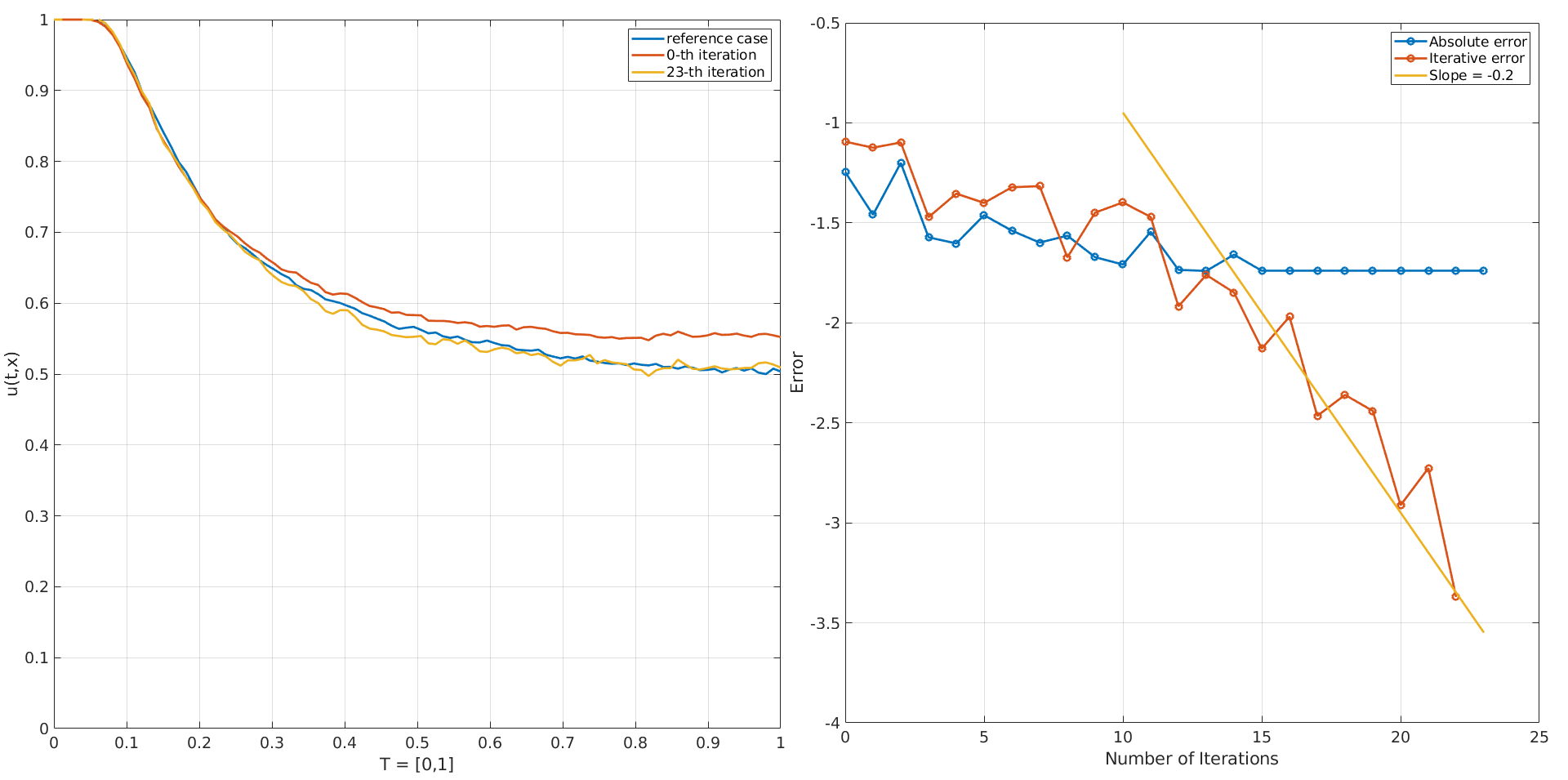}
\label{fig:cubicBoundedShiftd100}
\end{subfigure}
\caption{Trajectories of $u(t,x_{0})$ for $t \in [0,1]$, and plot of the error in $\log_{10}$ scale as a function of the number of iterations. Left block: cubic bounded  case \eqref{eq:polynomialbounded} in dimension $d = 10$ without the use of the shift. We see that the solution exhibits oscillations in time making the result quite inaccurate. Right block: cubic bounded  case \eqref{eq:polynomialbounded} in dimension $d = 100$ with the addition of the shift. Here we see that, even if the dimension is much larger than the figure on Left, the result is much more stable, and the final error decreases down to the value $0.02$. For both cases the initial condition has been taken as $\phi(x) = \mathds{1}_{\norm{x}_{2}\geq 1}$ and $x_{0} = \mathbf{e}$.}
\label{fig:cubicBoundedNoShiftd10+cubicBoundedShiftd100}
\end{figure}

\subsection{Spectral models}\label{subsec:spectral models}
In this section we present some numerical results relative to the proposed iterative method introduced in Section \ref{sec:iterationscheme}, by using spectral discretization. As stated at the beginning of this section, we will always assume the dimension $N$ of the underlining space to be one. Hence we will call $d$ the dimension of the discrete problem, that is, the number of Fourier modes that we are considering.
In this setup the matrix $A$ will then be a diagonal matrix, with entries $A_{k,k}  = -k^{2}$, and $A_{k,j} = 0$ for all $k \neq j$. Moreover we also take the value of $\sigma$, introduced at the beginning of Section \ref{sec:numericalResults} to be equal to one.

Regarding the choice of nonlinearities $B$ we have mainly in mind to improve the work done in \cite{FLR}, where the results are not completely satisfactory. We will show that, with the current modification given by the shift of the Gaussian process, the iterative scheme is able to tackle a broader class of problems.

We start by analyzing the case where $B$ is the following  polynomial vector field
\begin{equation}\label{eq:polynomialbounded}
B(x)_{i} = b_{0}\norm{\overline{y}}_{\infty}\frac{(\overline{y}_{i}-x_{i})\abs{\overline{y}_{i}-x_{i}}^{2}}{b_{0}\norm{\overline{y}}_{\infty}+\norm{\overline{y}-x}_{\infty}^{3}},\quad i=1,\dots,d
\end{equation}
where $\overline{y} \in \R^{d}$ is fixed to the value $2\mathbf{e}$, $\mathbf{e}$ being the vector with all components equal to one, and $b_{0} = 2$. This case consists in a standard cubic nonlinearity  but with the addition of a cut-off for large values of $\norm{x}$.  We highlight that, in this case, by the addition of the shift to the Gaussian process, the results of our numerical experiments are satisfactory up to dimension $d = 100$. See Figure \ref{fig:cubicBoundedNoShiftd10+cubicBoundedShiftd100} for a comparison with the same nonlinearity in dimension $d=10$ but without the use of the shift.

Next, we try to stress the method presented in this manuscript by testing some nonlinearities $B$ which fall outside of the present theory. We test the method for the vector field
\begin{equation}\label{eq:quadraticsimple}
B(x)_{i} = b_{0}(\overline{y}_{i}-x_{i})\abs{\overline{y}_{i}-x_{i}},\quad i=1,\dots, d
\end{equation}
without any kind of renormalization and $b_{0} = 1$. We remark that in order to prove convergence of the iteration scheme in its full generality, in \cite{FLR} we had to assume $B$ to be bounded. Hence, here we are going beyond that hypothesis by choosing $B$ not only to be unbounded, but of quadratic growth. This choice was driven by the desire to be able to tackle fluid dynamical problems where nonlinearities are typically of quadratic type. Of course here the situation is simpler since there is no mixing between the components, but it can be seen as a first step in that direction. In this case the results obtained by using the proposed algorithm are shown in Figure \ref{fig:QuadraticSimpleShifterr003d10}. Here we remark that we do not show the comparison with the case without the shift of the Gaussian process since in that case there is no convergence at all, and iterations blow up after a few steps.
\begin{figure}[t]
\centering
\includegraphics[width=0.7\textwidth]{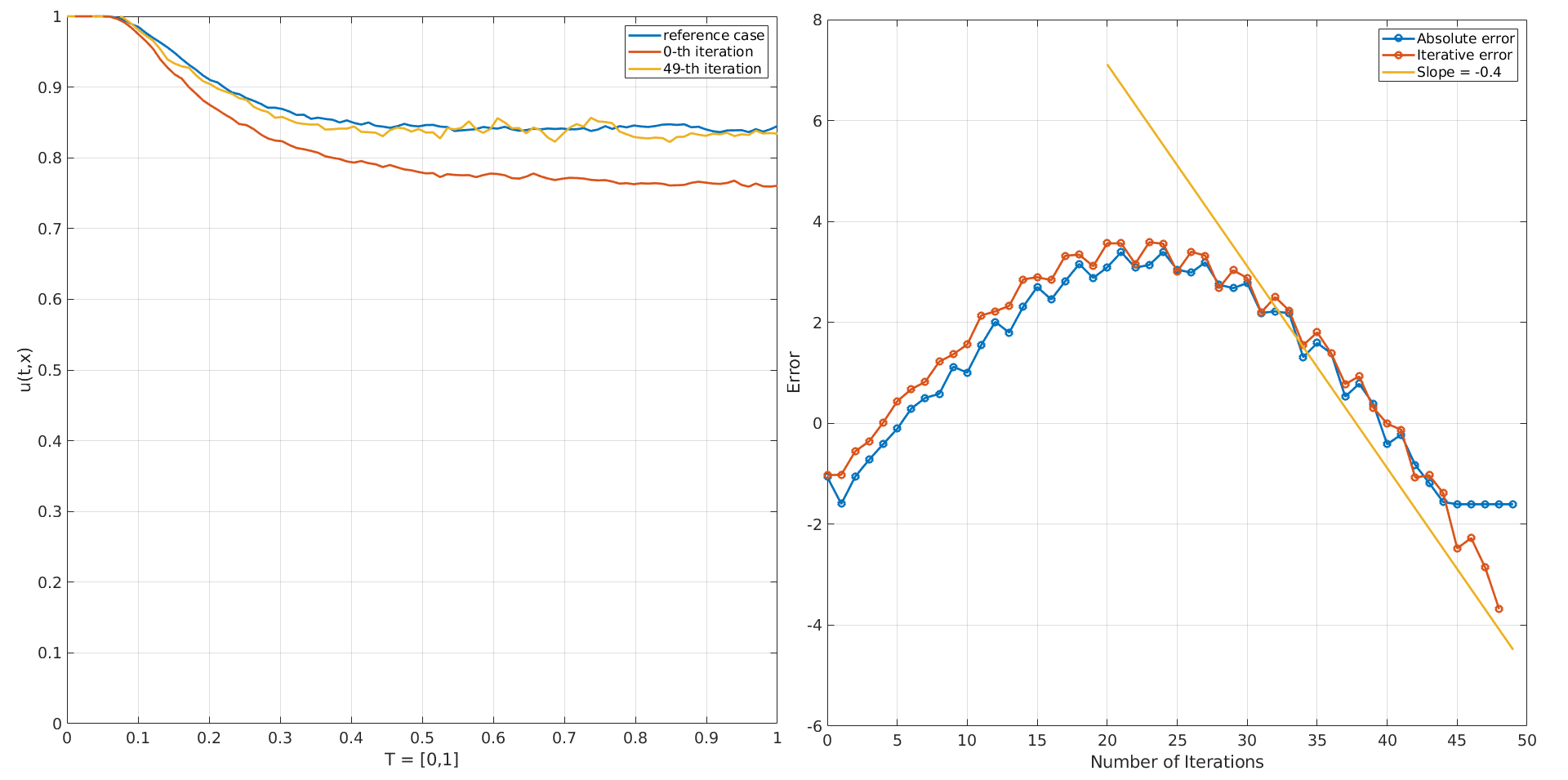}
\caption{Trajectories of $u(t,x_{0})$ for $t \in [0,1]$, and plot of the error in $\log_{10}$ scale as a function of the number of iterations. Strictly quadratic case \eqref{eq:quadraticsimple} in dimension $d = 10$. We see that the error first rises up and reaches values of order $10^{4}$ while at the final iteration is of order $10^{-2}$. The initial condition has been taken as $\phi(x) = \mathds{1}_{\norm{x}_{2}\geq 1}$ and $x_{0} = \mathbf{e}$. }
\label{fig:QuadraticSimpleShifterr003d10}
\end{figure}

Finally we test the method on an even more difficult situation. As we said the final aim of this research is to go into the direction of solving problems related to fluid dynamics (e.g. in climate studies). Hence we select a simple but profound approximation of Fourier modes for the Navier-Stokes equations: the \emph{Dyadic model}. We consider
\begin{equation}\label{eq:dyadicmodel}
\left\{ \aligned
B(x)_{1} &= F_{1} -k_{1}x_{1}x_{2},\\
B(x)_{i} &= k_{i-1}x_{i-1}^{2}-k_{i}x_{i}x_{i+1},\quad i=2,\dots,d-1,\\
B(x)_{d} &= k_{d-1}x_{d-1}^{2},
\endaligned \right.
\end{equation}
where $k_{i} = \lambda^{2i}$ for $i=1,\dots,d$ and $F_{1}$ is a positive constant. In this specific case, we choose the matrix $A$ to be diagonal with diagonal entries $A_{i,i} = -k_{i}$. The rationale of the previous choice is the following: energy is transferred from lower Fourier modes (corresponding to low index components) to higher ones. The forcing term $F_{1}$ is added in order to insert new energy into the system, while the condition for $i=d$ is just the boundary condition for dealing with a finite number of modes instead of an infinite one, see \cite[Chapter 3]{FSaintFlour} for a full discussion.

This model still behaves like that in \eqref{eq:quadraticsimple}: it is of quadratic growth but is much more complicated, since it involves transportation of energy between modes and many other phenomena.
In this case we test the method only in $d = 10$, see Figure \ref{fig:dyadicd=10u0mean}. The results are only partially good: at first we see an improvement of the error with respect to the reference solution on the first iteration, which then degenerates in further approximations. However we point out that this result is a first step in the direction of fluid dynamical problems, and we can obtain such a result only by applying the method described in the current paper. On the contrary, the simpler method described in \cite{FLR} was only able to treat less relevant problems.

\begin{figure}[t]
\centering
\includegraphics[width=0.7\textwidth]{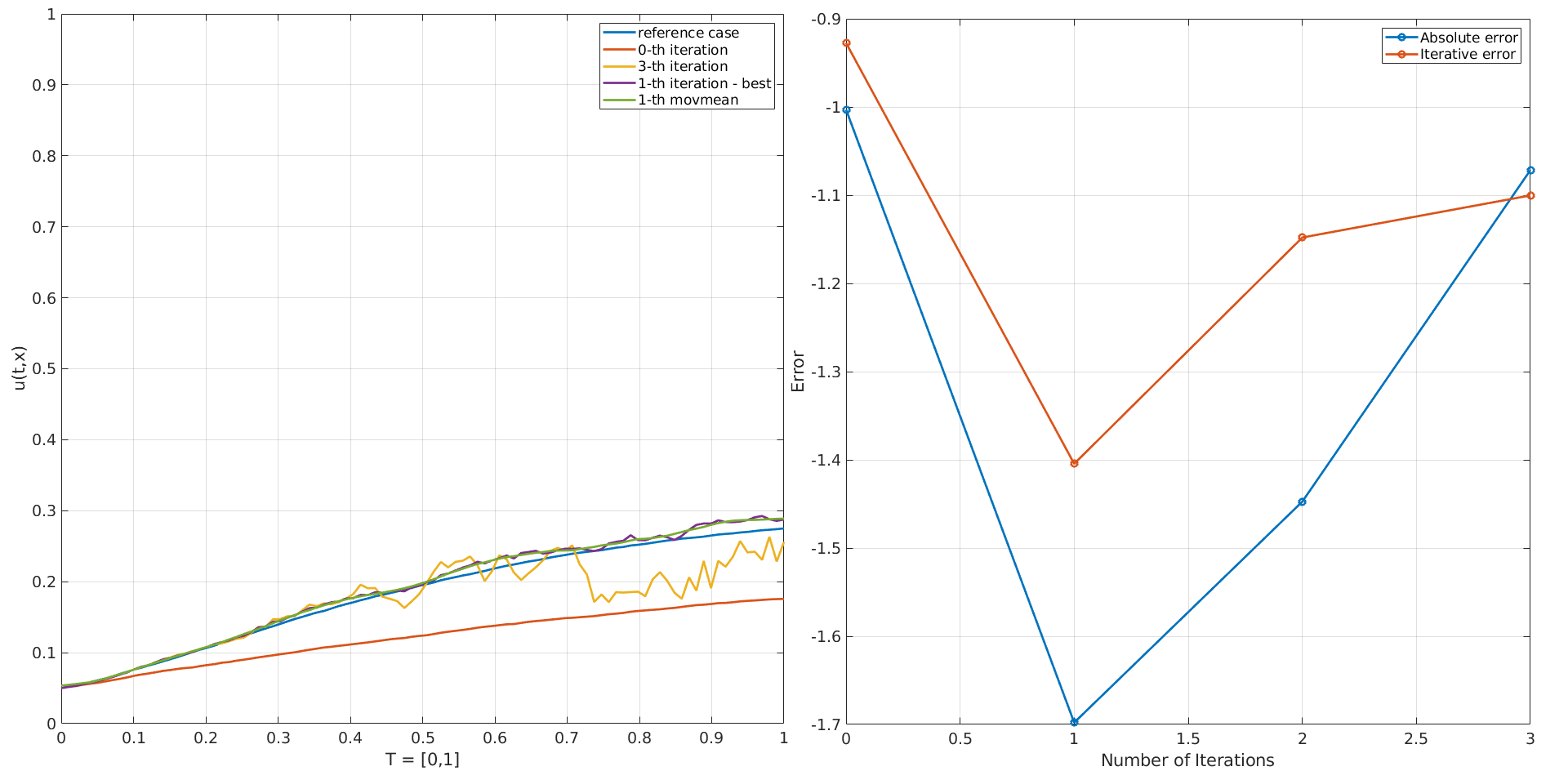}
\caption{Trajectories of $u(t,x_{0})$ for $t \in [0,1]$, and plot of the error in $\log_{10}$ scale as a function of the number of iterations. Dyadic case \eqref{eq:dyadicmodel} in dimension $d = 10$. The green line labeled by \emph{movmean} means that a time average as been applied to smooth the solution. We see that the first iteration improve the result, even if further iterates degenerate. The initial condition has been taken as $\phi(x) = \frac{1}{d} \sum_{i=1}^{d}x_{i} $, the value of $\lambda$ is set to $1.1$, $F_{1} = 2$ and $x_{0}  = \mathbf{e}_{1}$. }
\label{fig:dyadicd=10u0mean}
\end{figure}

\section{The probability distribution}\label{sec:visualization}

When working with SDEs, sometimes we are interested in the numerical value of
some specific expected value or probability, which is investigated in the
previous section. Sometime else we would like to have a graphical
representation of the probability law. Kolmogorov equation a priori does not
seem to be the right tool for such a purpose; the best one looks the Fokker-Planck equation. We
could rewrite the Fokker-Planck equation as a modified Kolmogorov equation with
a zero-order term and devise an iterative Gaussian approximation similarly to
above. However, such approach is not immediately useful, since it aims to
compute the pointwise values of a probability density in a very high
dimensional space, facing the problem of how many points and which one should
be computed to have interesting informations.

In addition, we should not forget that a visualization requires projecting the
solution on $k$-dimensional space with $k=1,2$ (at most $k=3$ if isosurfaces
are plotted). It is not clear at all that solving Fokker-Planck equation is
the right way to get such information.

One alternative is to choose the low-dimensional projection, say $k=2$ to fix
the ideas. We divide the plane in $N$ sets and compute the probability that
$X^{x_0}_{t}$ takes values in each one of these sets by means of $N$ solutions of
Kolmogorov equation. Since only the function $\phi$ changes, one can rearrange
the numerical code saving most of the computations;\ the final cost is not so
larger than the cost of a single simulation of Kolmogorov equation, if $N$ is
not too large. We have done simulations of this kind and the results are good
but looks very approximate, due to the moderate cardinality $N$ of the
partition (see Figure \ref{fig:2dHist+surfSmooth} with $N=1000$).
\begin{figure}[t]
\begin{subfigure}{0.49\textwidth}
\includegraphics[width=\textwidth]{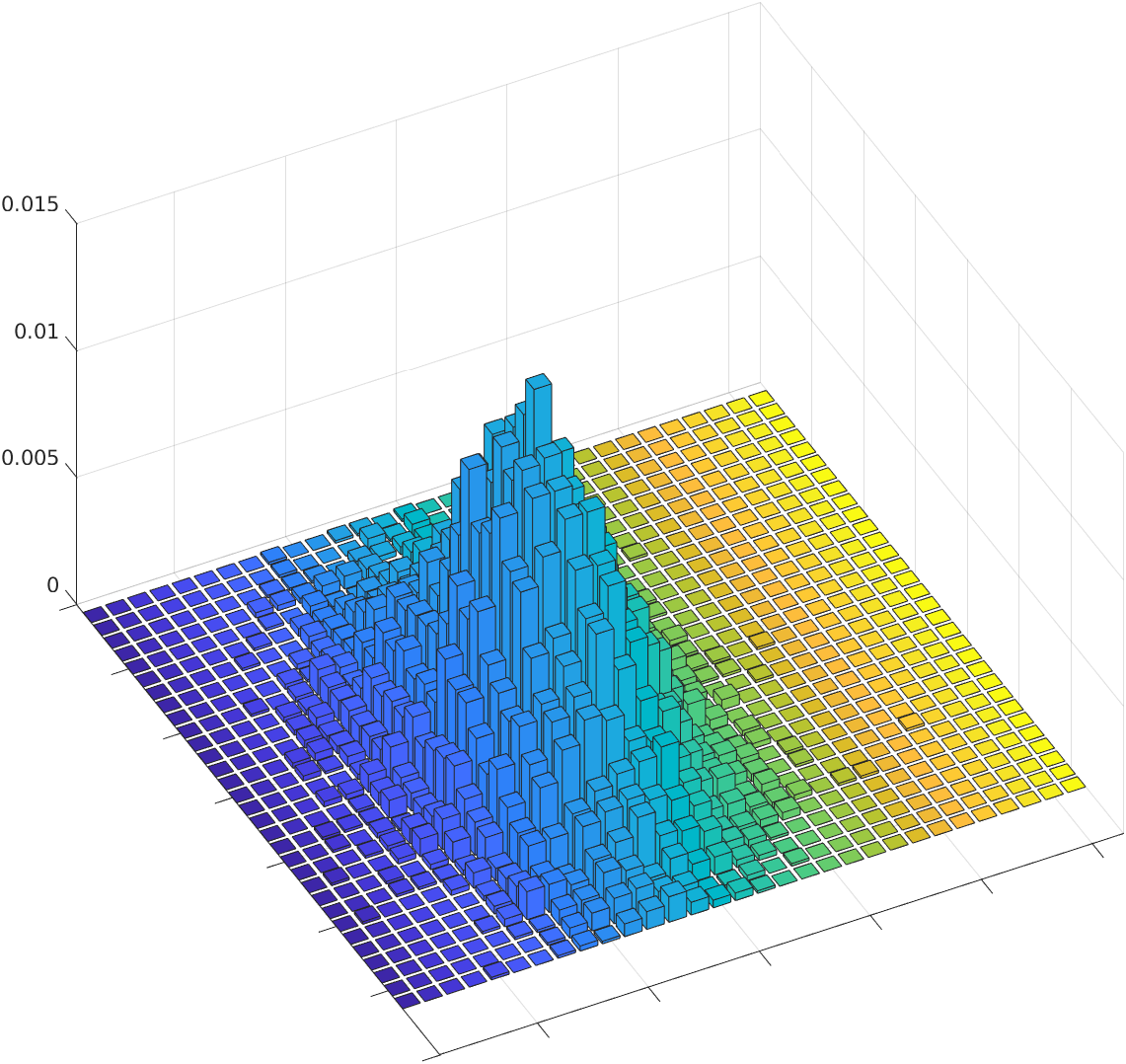}
\label{fig:2dHist}
\end{subfigure}
\begin{subfigure}{0.49\textwidth}
\includegraphics[width=\textwidth]{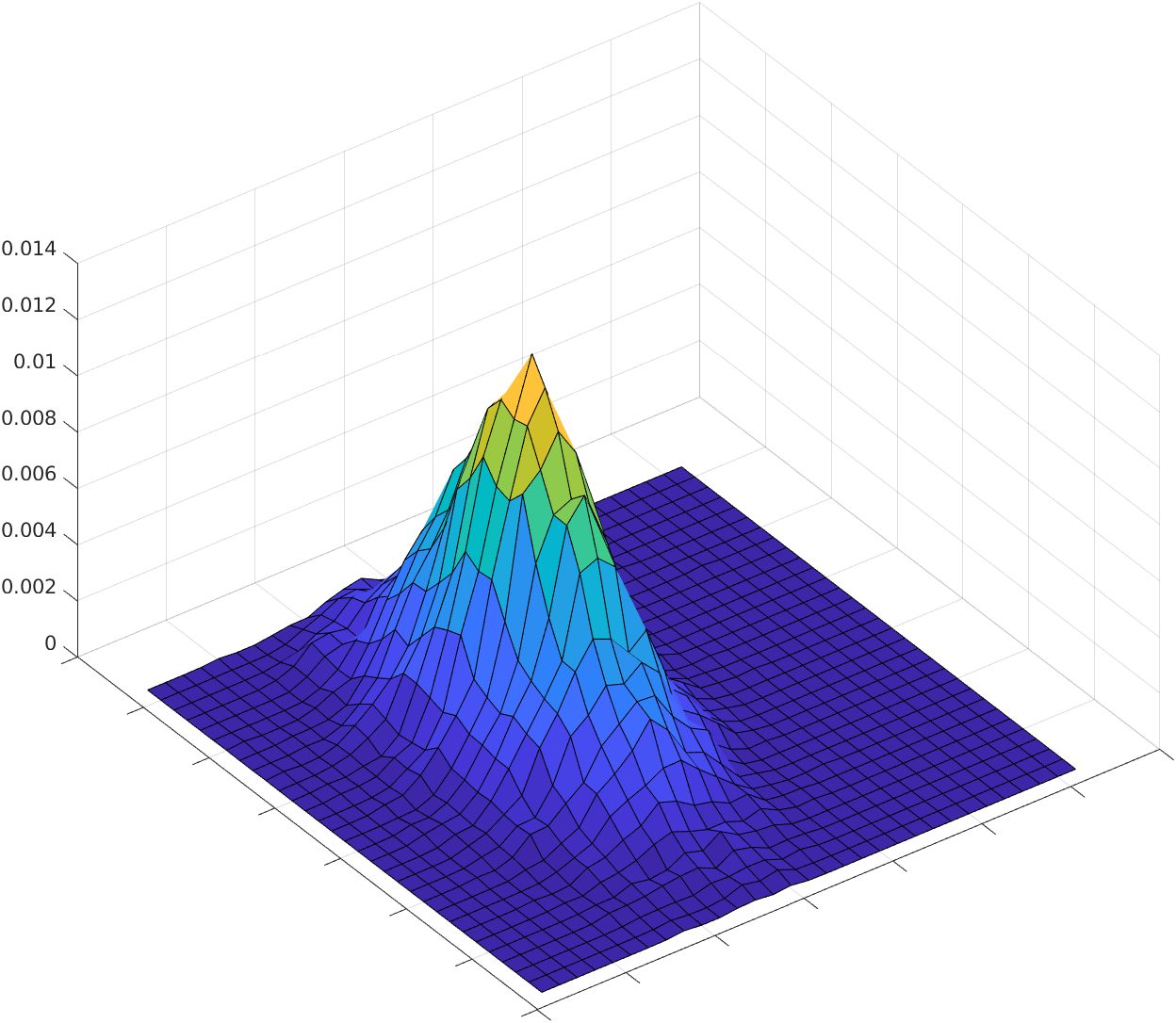}
\label{fig:surfSmooth}
\end{subfigure}
\caption{First two Fourier components in the cubic bounded case \eqref{eq:polynomialbounded}. Each square grid has a height corresponding to the probability for the variable $X^{x_0}_{t}$ to be inside that square. Each values has been computed by means of the approximation scheme described in Section \ref{sec:iterationscheme} by choosing the function $\phi$ as the indicator function of the corresponding square grid. In the figure on the right a density-like plot instead of a histogram has been used to give a better feeling of the resulting probability distribution.}
\label{fig:2dHist+surfSmooth}
\end{figure}

Next, in order to understand the mechanism behind the Gaussian approximation obtained via Kolmogorov equation, we rewrite the approximants $v^n(t,x_0)= v^{n}_{0,t}(x_0)$ as in Remark \ref{remark:iterativeIn} and approximate the mean by taking averages over independent samples:
\[
v^{n}(t,x_{0}) = \E\big[\phi( Z_{t}^{x_{0}}) I^{n}_{t}(x_{0}) \big] \approx \frac{1}{N_{s}} \sum_{i=1}^{N_{s}} \phi\Big( Z^{x_{0},i,\Delta_q t}_{j} \Big) I^{n,x_{0},i,\Delta_q t}_{j} .
\]
Moreover, we also have
\begin{equation}
\label{eq:reweighting}
u^{n}(t,x_{0}) = \sum_{k=0}^{n} v^{k}(t,x_{0})
= \E\bigg[\phi( Z_{t}^{x_{0}}) \sum_{k=0}^{n} I^{k}_{t}(x_{0}) \bigg] \approx \frac{1}{N_{s}} \sum_{i=1}^{N_{s}} \phi\Big( Z^{x_{0},i,\Delta_q t}_{j} \Big) \bigg[\sum_{k=0}^{n} I^{k,x_{0},i,\Delta_q t}_{j}\bigg].
\end{equation}
This last expression has to be compared to that related to classical Monte Carlo method applied to the nonlinear process $X^{x_0}_{t}$:
\[
u(t,x_{0}) = \E{\phi(X_{t}^{x_{0}}) } = \frac{1}{N_{s}} \sum_{i=1}^{N_{s}} \phi\Big( X^{x_{0},i,\Delta_q t}_{j} \Big).
\]
Hence, we see the parallelism between the two: the iteration scheme obtained by Kolmogorov equation produces samples from a Gaussian process $Z_{t}^{x_{0}}$ that should be reweighted by means of the iterated integral expression $\sum_{k=0}^{n} I^{k}_{t}(x_{0})$. This sum expresses the order of accuracy of the weighting procedure, by having $I^{0}_{t}(x_{0}) \equiv 1$ and leading to the limit $u^{n}(t,x_{0}) \to u(t,x_{0})$ when $n$ tends to infinity. From this interpretation we see how the Gaussian process $Z^{x_0}_{t}$ represents the core of the initial approximation of the solution $X^{x_0}_{t}$. The relations with the classical topic of importance sampling for Monte Carlo method is also evident (see \cite{Importance} for an introduction): to compute quantities related to a ``difficult'' distribution such as that of $X^{x_0}_{t}$, we use samples from a simpler one such as $Z^{x_0}_{t}$ and adapt the probability measure.

In the sequel we investigate what are the effects of the shift of the Gaussian process introduced at the beginning of the paper. As we said the Gaussian process can be seen as a first guess of $X^{x_0}_{t}$. Hence we analyze this initial approximation by comparing samples coming from the processes $Z^{x_0}_{t}$ and $X^{x_0}_{t}$. As stated before, in order to visualize our problem, we have to restrict the available information and project every sample on a two dimensional space. In Figure \ref{fig:cloud1+cloud2} we compare samples of only the first two components from the cubic bounded case \eqref{eq:polynomialbounded} obtained with or without shifting the Gaussian process.  We can see that the initial approximation of $X^{x_0}_{t}$ obtained with shift is much better than that without the shift.
\begin{figure}[t]
\begin{subfigure}{0.49\textwidth}
\includegraphics[width=\textwidth]{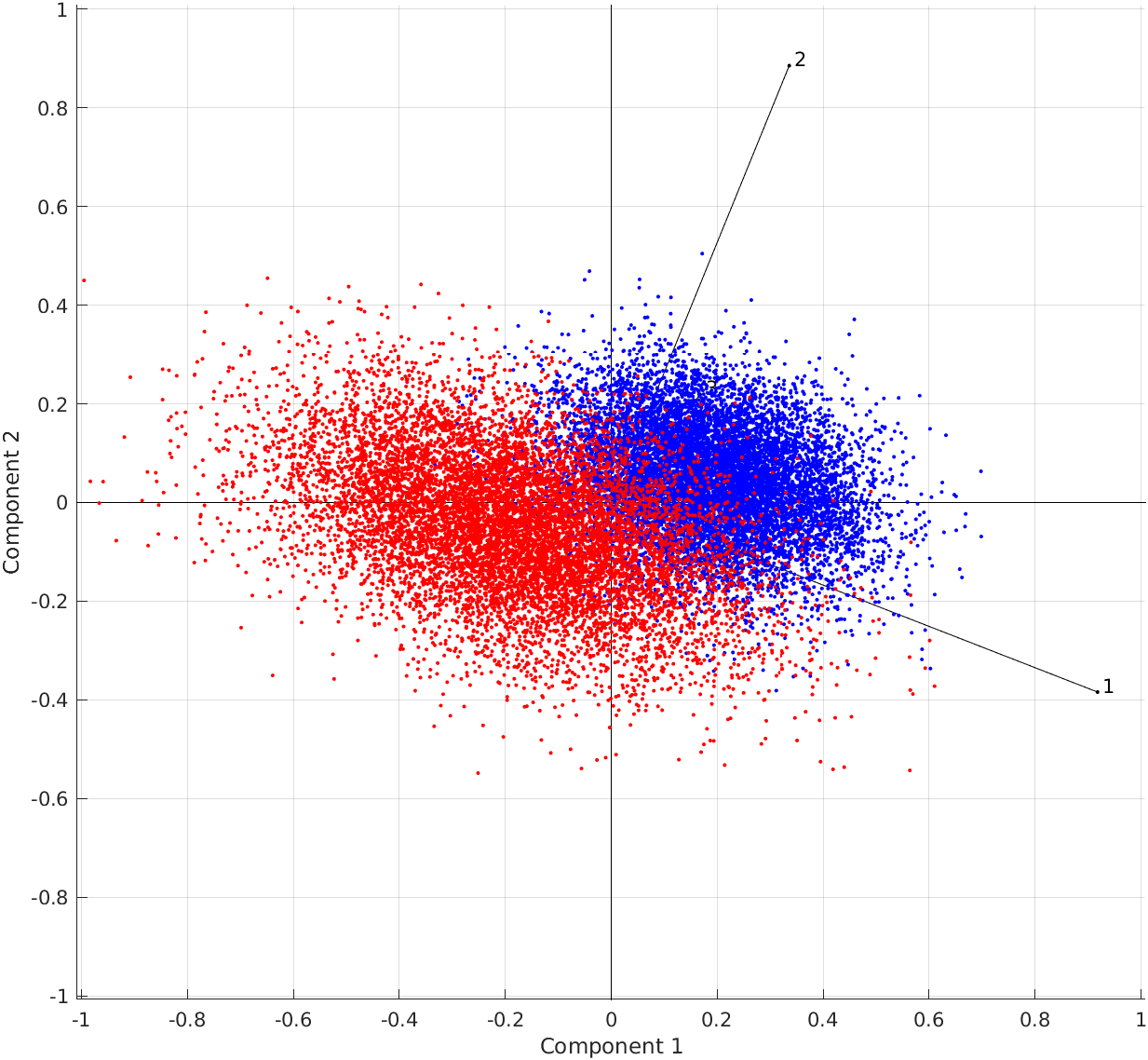}
\label{fig:cloud1}
\end{subfigure}
\begin{subfigure}{0.49\textwidth}
\includegraphics[width=\textwidth]{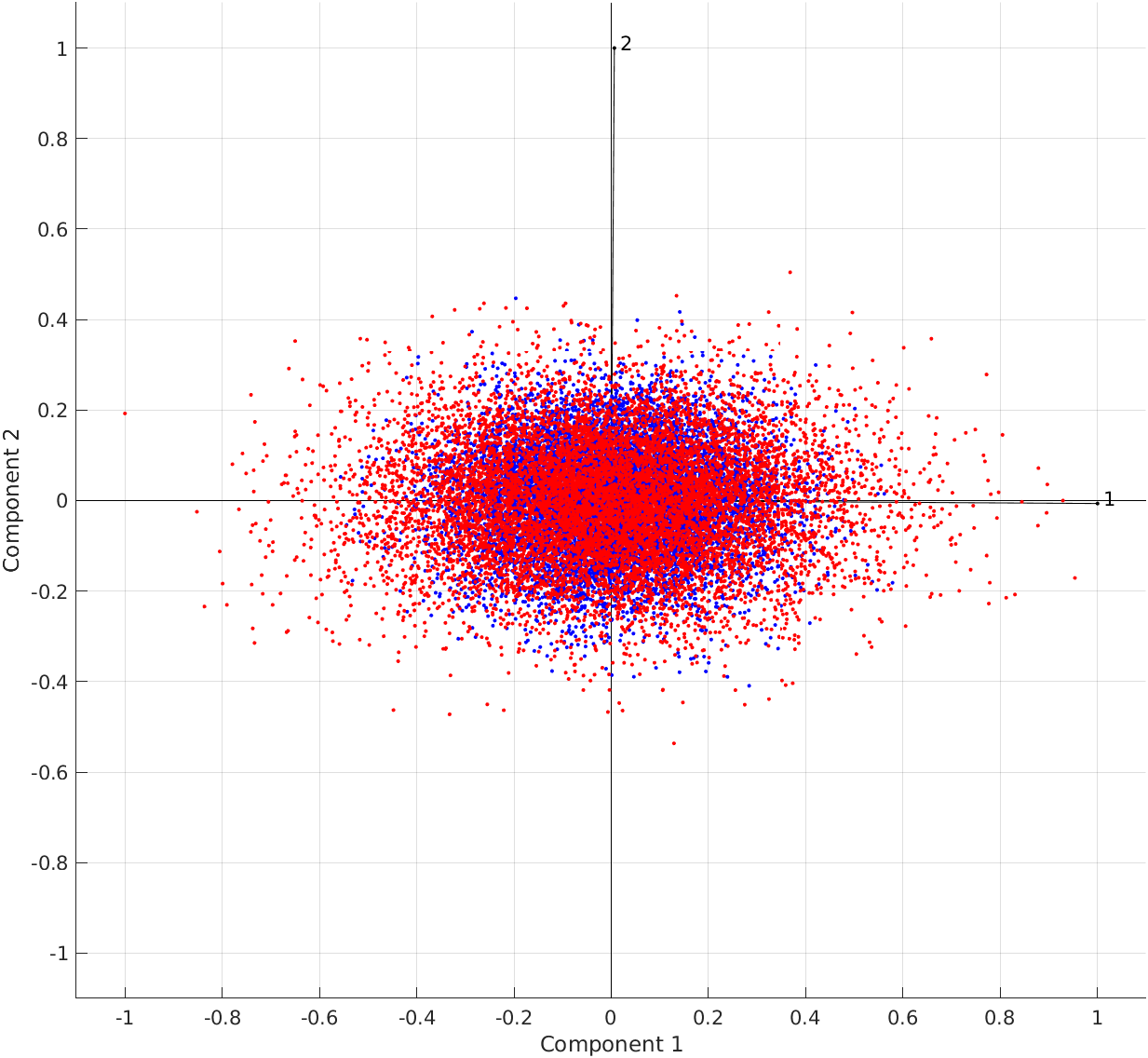}
\label{fig:cloud2}
\end{subfigure}
\caption{First two Fourier components in the cubic bounded case \eqref{eq:polynomialbounded}. Each point correspond to an independent realization of the nonlinear process $X^{x_0}_{t}$ (blue) or of the linear one $Z^{x_0}_{t}$ (red). On the left we see the result by computing samples without the shifting the gaussian process while on the right we see the effects of the shift. As we can see on the right the samples coming from the gaussian process with shift are much closer to the samples of non linear process, providing a better initial approximation for the iteration scheme.}
\label{fig:cloud1+cloud2}
\end{figure}

As we saw in Section \ref{subsec:spectral models} the closeness of the initial approximation is the key to obtain a good approximation through iterations. This can be appreciated even further by the use of histograms. In Figure \ref{fig:noshift3} we show the distributions of the processes involved by projecting only on the first Fourier component. The yellow histogram has been obtained by using samples coming from the Gaussian process $Z^{x_0}_{t}$ by giving different
weight to each samples following expression \eqref{eq:reweighting}. We see that the addition of the shift to the initial Gaussian approximation makes the final result much better than in the case without any shift. In particular in Figure \ref{fig:noshift3} we see how the histogram colored in yellow, obtained by the Gaussian process plus the reweighting process, is not able to cover the entire support of the target distribution in blue. This is not the case in Figure \ref{fig:shift2+shift3} where, since the Gaussian distribution (red) is able to cover the empirical support of the nonlinear process (blue), the approximation is of higher precision. The same phenomenon can be observed in Figure \ref{fig:cloud1+cloud2}.
\begin{figure}[t]
\centering
\includegraphics[width=0.5\textwidth]{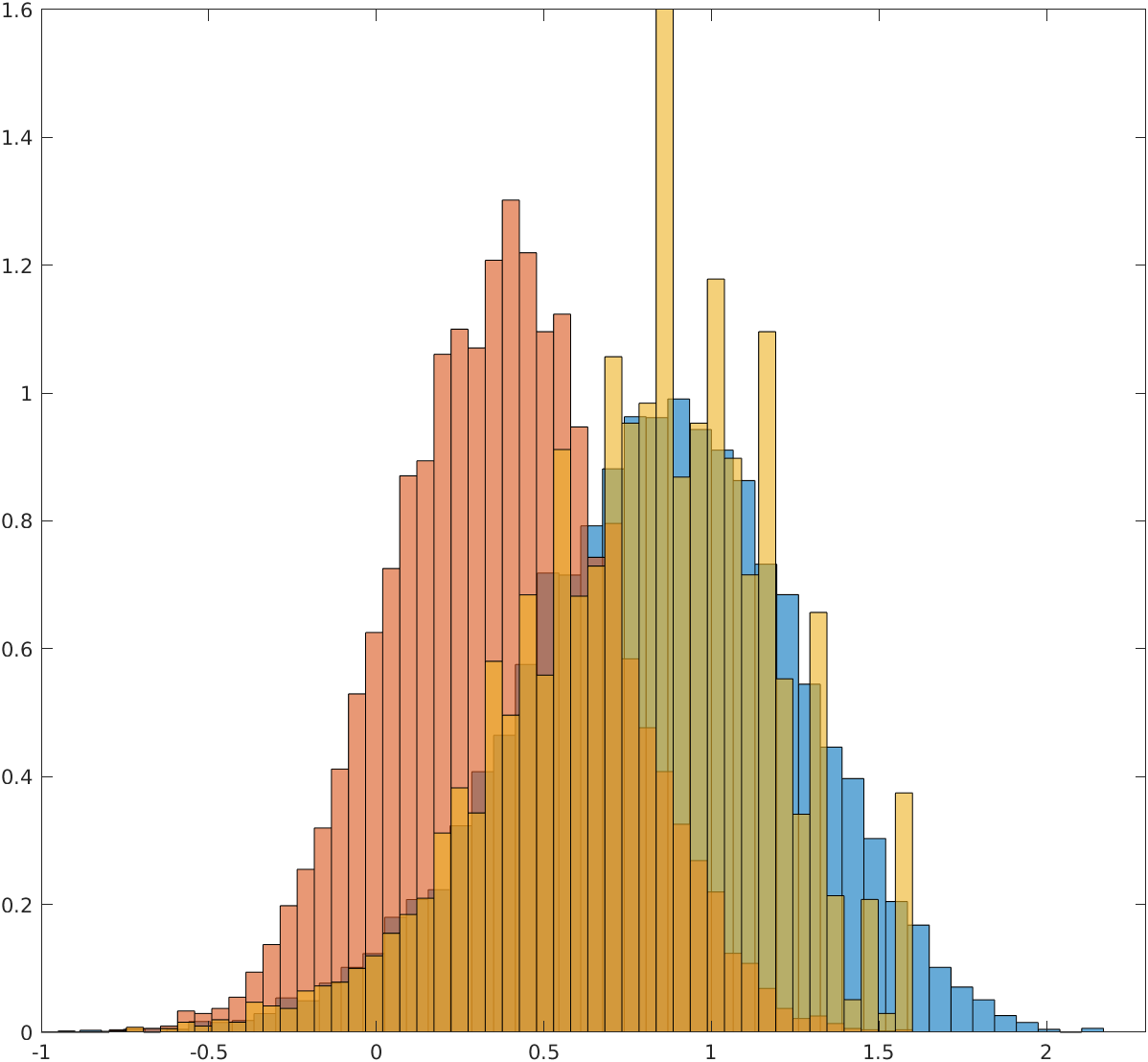}
\caption{Histograms of the first Fourier component of the nonlinear process $X^{x_0}_{t}$ (blue), Gaussian process without shift $Z^{x_0}_{t}$ (red), and Gaussian process after the operation of reweighting \eqref{eq:reweighting} (yellow), in the polynomial bounded case \eqref{eq:polynomialbounded}.}
\label{fig:noshift3}
\end{figure}

However, the previous types of plots are not able to grasp the difficulties of more intricate problems. In particular, in the example of Dyadic model of turbulence \eqref{eq:dyadicmodel}, not only is the problem nonlinear, but there is also high transfer of energy between modes, projecting on only one component is not enough to catch the whole phenomenon. Hence we produced samples from the nonlinear process and Gaussian process in dimension $d=20$ and performed \emph{principal component analysis} to visualize the samples. The results are presented in Figure \ref{fig:banana1+banana2}. We can see that, whether taking advantage of the shift or not, the initial approximations given by the Gaussian processes are very poor. In fact the particular shape of the nonlinear process doesn't allow an appropriate reconstruction of its distribution. This is perhaps a limitation of the current method that we hope to improve in a future research.

\begin{figure}[t]
\begin{subfigure}{0.49\textwidth}
\includegraphics[width=\textwidth]{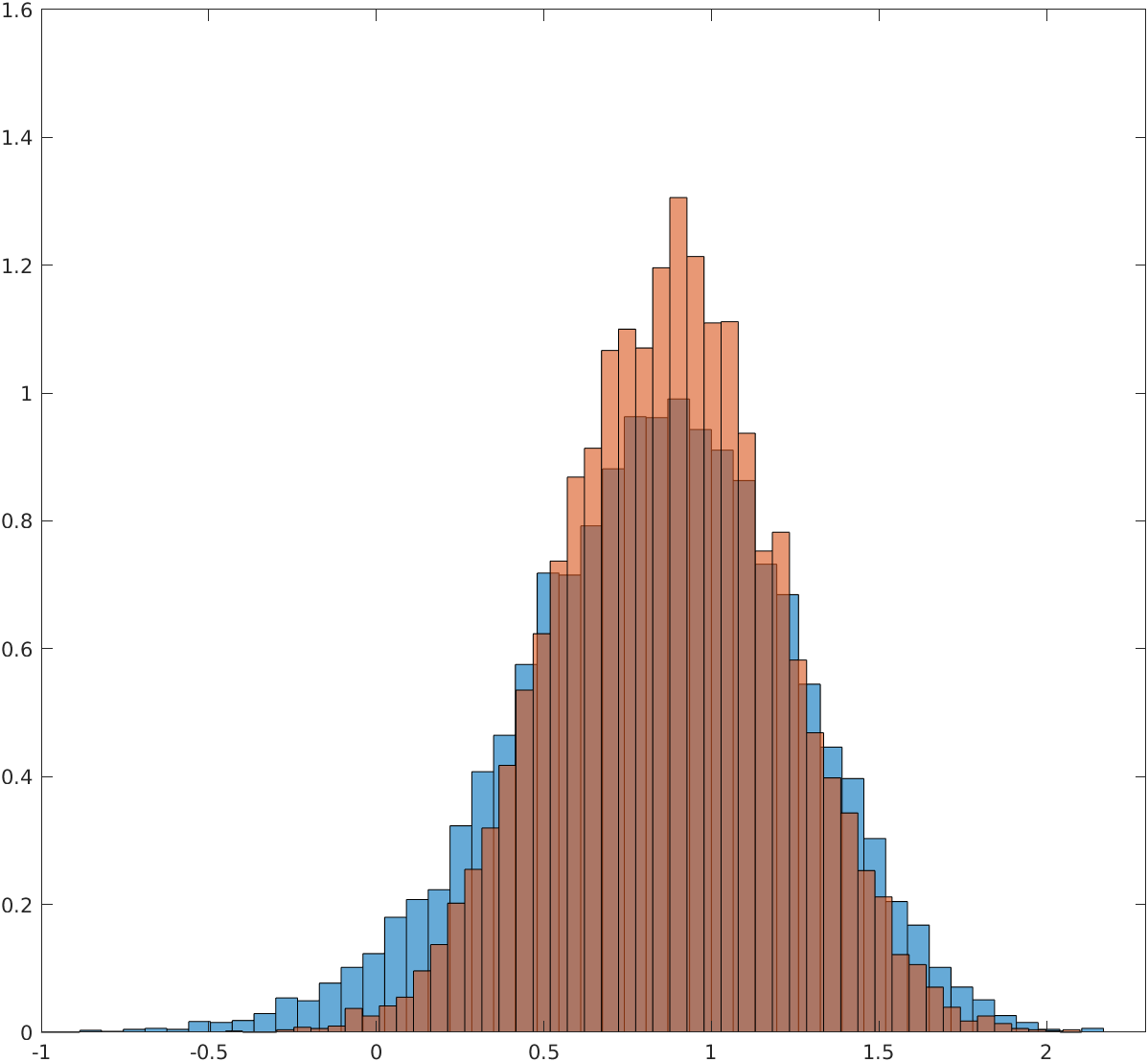}
\label{fig:shift2}
\end{subfigure}
\begin{subfigure}{0.49\textwidth}
\includegraphics[width=\textwidth]{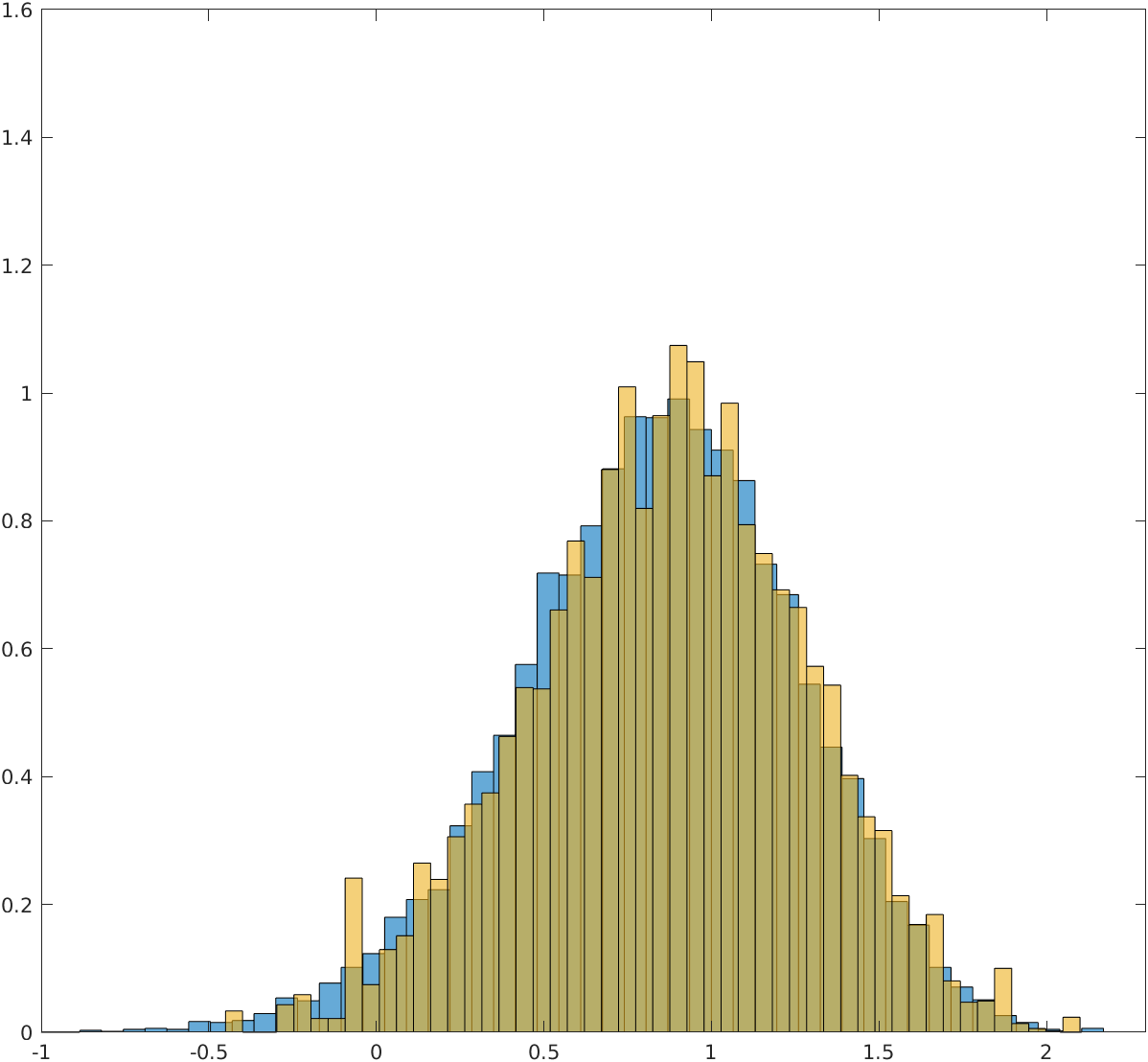}
\label{fig:shift3}
\end{subfigure}
\caption{Histograms of the first Fourier component of the nonlinear process $X_{t}$ (Blue), Gaussian process with shift $Z_{t}$ (red), and Gaussian process after the operation of reweighting \eqref{eq:reweighting} (yellow), in the polynomial bounded case \eqref{eq:polynomialbounded}. Differently from the case without shift we see that the histograms in blue and yellow are almost superimposed.}
\label{fig:shift2+shift3}
\end{figure}

\begin{figure}[t]
\begin{subfigure}{0.49\textwidth}
\includegraphics[width=\textwidth]{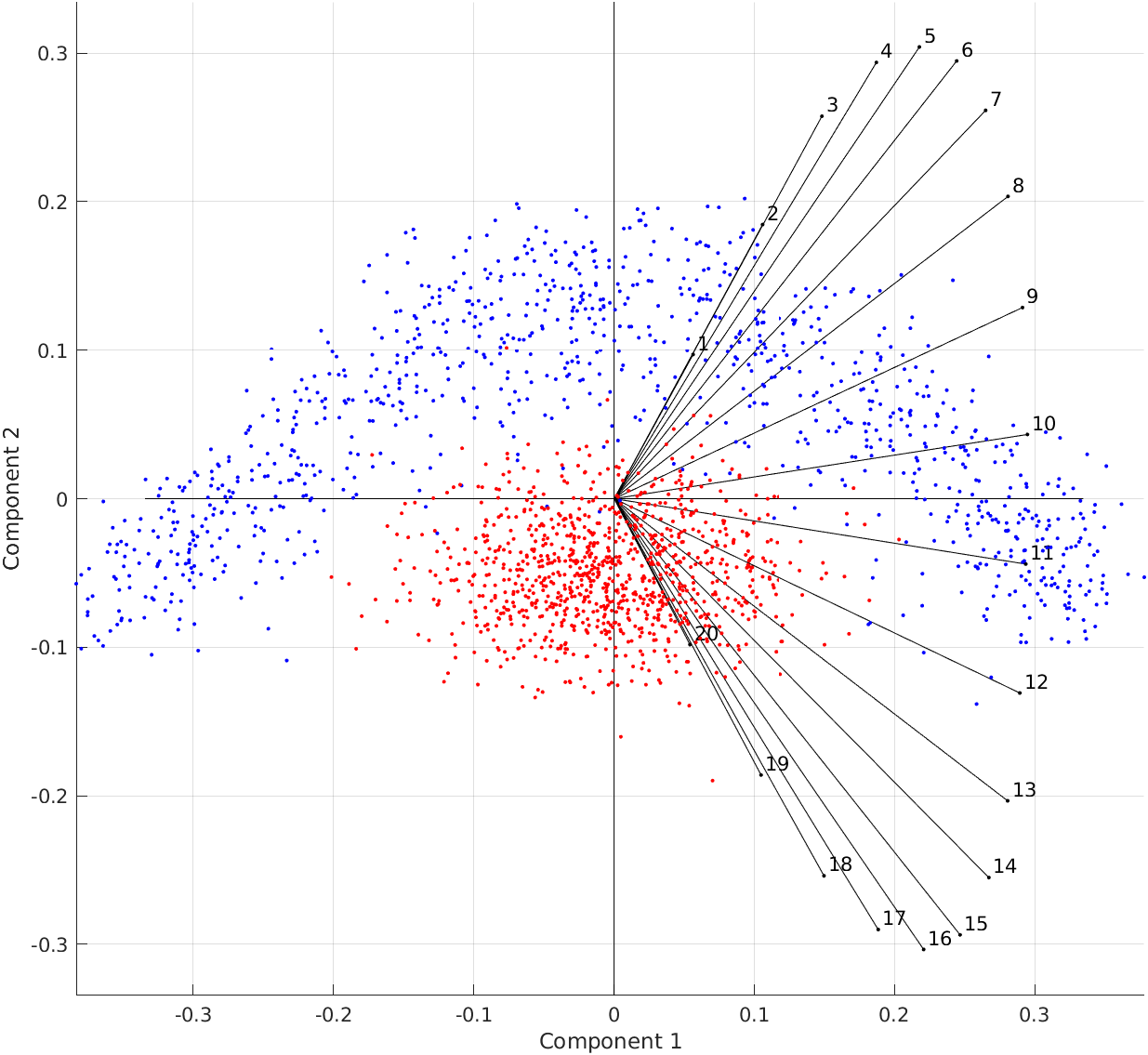}
\label{fig:banana1}
\end{subfigure}
\begin{subfigure}{0.49\textwidth}
\includegraphics[width=\textwidth]{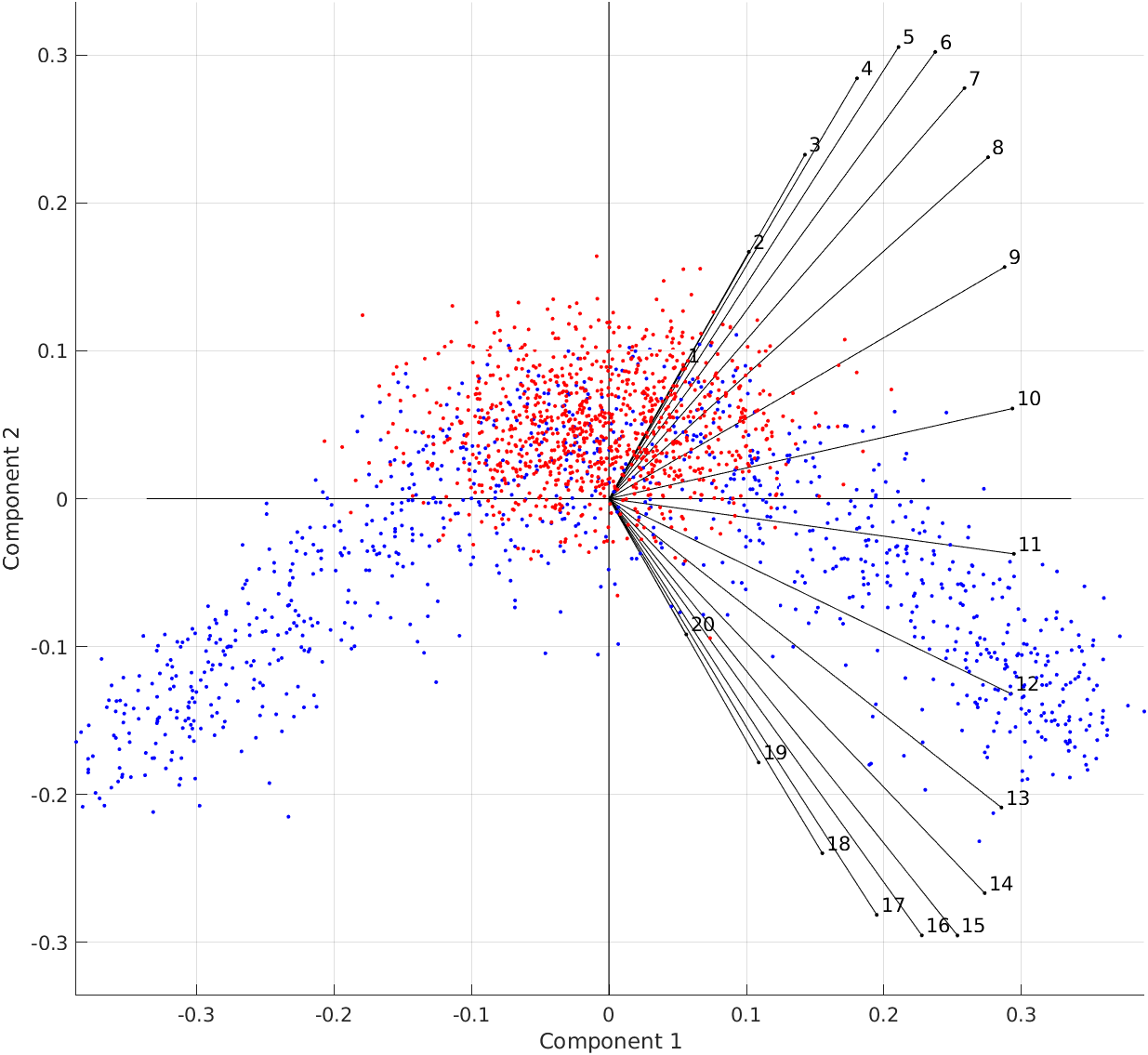}
\label{fig:banana2}
\end{subfigure}
\caption{Principal component analysis for the samples of the Dyadic model \eqref{eq:dyadicmodel} in dimension $d = 20$. The nonlinear process $X_{t}$ is colored in blue, while the Gaussian process in red. On the left we see the result without the addition of the shift on the Gaussian process, while it is used on the right hand side. We see that in both pictures, the initial approximation given by the Gaussian, is too poor to allow an accurate reconstruction of the target distribution.}
\label{fig:banana1+banana2}
\end{figure}

\section{Appendix}

For readers' convenience, let us provide some proofs of the theoretical
claims made above. We first prove Lemma \ref{thm-1} for which we need some
preparations. Recall that we work on the Euclidean space $\mathbb{R}^{d}$, and
$A$ (resp. $Q$) is a self-adjoint and strictly negative (resp. positive)
matrix of order $d$. Let $F\in C\big([0,T]\times\mathbb{R}^{d}, \mathbb{R}^{d}
\big)$ be a time-dependent continuous vector field; we consider the second
order differential operator: for $f\in C_{b}^{2}(\mathbb{R}^{d})$,
\[
L_{t} \phi(x) = \frac12 \mathrm{Tr}(Q D^{2} \phi)(x) + \<Ax + F(t,x),
D\phi(x)\>, \quad(t,x ) \in[0,T]\times\mathbb{R}^{d}.
\]
Assume that the following SDE
\[
\mathrm{d} X_{t}= (AX_{t} + F(t,X_{t}))\,\mathrm{d} t + \sqrt{Q}\, \mathrm{d}
W_{t}, \quad t\geq s,\ X_{s}=x
\]
has a unique strong solution denoted by $X^{x}_{s,t}$; let $\{T_{s,t}\}_{0\leq
s < t\leq T}$ be the associated Markov semigroup:
\[
T_{s,t} \phi(x)= \mathbb{E} \phi(X^{x}_{s,t}), \quad x\in\mathbb{R}%
^{d},\ 0\leq s<t \leq T,\ \phi\in\mathcal{B}(\mathbb{R}^{d}).
\]
Here $\mathcal{B}(\mathbb{R}^{d})$ is the space of bounded measurable
functions on $\mathbb{R}^{d}$. There are more general results than the
following one, but it is sufficient for our purpose (see \cite{DPL, Krylov,
KryPri} for some other related results).

\begin{lemma}
\label{append-lem-1} Assume that $F\in C^{0,1}_{b}\big([0,T]\times
\mathbb{R}^{d}, \mathbb{R}^{d} \big)$. Then for any $\phi\in\mathcal{B}%
(\mathbb{R}^{d})$, one has $T_{s,t} \phi\in C_{b}^{2}(\mathbb{R}^{d})$ for all
$0\leq s <t \leq T$, and the following equation holds in the classical sense:
\[
\frac{\partial}{\partial s} T_{s,t} \phi(x)= -L_{s} (T_{s,t} \phi
)(x),\quad(s,x)\in[0,t)\times\mathbb{R}^{d}.
\]
Moreover, if $\phi\in C_{b}^{2}(\mathbb{R}^{d})$, then we also have
\[
\frac{\partial}{\partial t} T_{s,t} \phi(x)= T_{s,t} (L_{t}\phi)(x),\quad
(t,x)\in(s,T]\times\mathbb{R}^{d} .
\]

\end{lemma}

\begin{proof}
Applying the It\^o formula to $\phi(X^x_{s,t})$, we immediately obtain the last assertion from the definition of the operator $L_t$; the details are omitted here. The proofs of other assertions can be found in \cite[Section 9.4.3]{DPZ}, where the authors deal with the infinite dimensional case with a time-independent nonlinear drift $F$.  Note that, in the current finite dimensional setting, the matrices $A$ and $Q$ are bounded operators, and it is easy to see that the assumptions like \cite[(9.48) and (9.50)]{DPZ} are satisfied. See \cite[Theorem 3.2.1]{Str-Var} for a related result, but there the drift coefficient is assumed to be bounded.
\end{proof}

Now we can give the

\begin{proof}[Proof of Lemma \ref{thm-1}]
Recall the semigroups $\{P_{s,t}\}_{0\leq s < t\leq T}$ and $\{S_{s,t}\}_{0\leq s < t\leq T}$ defined in the introduction. First, since $B_0\in C^{0,1}_b \big([0,T]\times \R^d\big)$, for any $\phi\in \mathcal{B}(\mathbb{R}^{d})$ and $0\leq s<t \leq T$, we have by Lemma \ref{append-lem-1} that $P_{s,t}\phi \in C_b^2(\R^d)$ and it solves
$$\frac{\partial}{\partial s} P_{s,t}\phi(x) + \frac12 {\rm Tr}(Q D^2 P_{s,t}\phi)(x) + \<Ax + B_0(s,x), D P_{s,t}\phi(x)\> =0, \quad x\in \R^d. $$
Next, the shift function $f\in C\big([0,T],\R^d\big)$ can be viewed as a space independent vector field on $\R^d$, thus for any $\varphi \in C_b^2(\R^d)$, the last assertion of Lemma \ref{append-lem-1} implies that
$$\frac{\partial}{\partial t} S_{s,t}\varphi(x) =S_{s,t}\bigg[ \frac12 {\rm Tr}(Q D^2\varphi)+ \<A\cdot + f(t), D\varphi\>\bigg](x) , \quad x\in \R^d.$$
Now for any fixed $0\leq s<t\leq T$ and $\phi\in \mathcal{B}(\mathbb{R}^{d})$, we consider $S_{s,r}(P_{r,t} \phi)$ for $r\in (s,t)$. Since $P_{r,t} \phi\in C_b^2(\R^d)$ for all $r<t$, we deduce from the above preparations that
$$\aligned
\frac{\partial}{\partial r} \big[S_{s,r}(P_{r,t} \phi)(x)\big] &= S_{s,r}\bigg[\frac12 {\rm Tr}(Q D^2P_{r,t}\phi)+ \<A\cdot + f(r), DP_{r,t}\phi\>\bigg](x) \\
&\quad + S_{s,r}\bigg[-\frac12 {\rm Tr}(Q D^2 P_{r,t}\phi) -\<A\cdot + B_0(r,\cdot), D P_{r,t}\phi\>\bigg](x).
\endaligned $$
Recalling that $B_0(r,x) = B(r,x) + f(r)$, we arrive at
$$\frac{\partial}{\partial r}\big[S_{s,r}(P_{r,t} \phi)(x)\big] =  - S_{s,r} \<B(r,\cdot), D P_{r,t}\phi\> (x).$$
Finally we integrate the variable $r$ on the interval $[s,t]$ and obtain
$$S_{s,t}\phi(x) = P_{s,t}\phi(x) - \int_s^t S_{s,r} \<B(r,\cdot), D P_{r,t}\phi\> (x)\,\d r$$
which is nothing but the formula in Lemma  \ref{thm-1}.
\end{proof}

In the remainder of this part we prove Lemma \ref{lem-derivative}.

\begin{proof}[Proof of Lemma \ref{lem-derivative}]
We follow the idea of \cite[Proposition 2.28]{DaPrato}. Note that, by \eqref{modified-Gaussian}, the random variable $Z^x_{s,t}$ has the Gaussian law $N_{e^{(t-s)A} x + F_{s,t}, Q_{t-s}}$ with the center $e^{(t-s)A} x + F_{s,t}\in \R^d$ and covariance matrix $Q_{t-s}$, thus
$$S_{s,t}\phi(x) = \E [\phi(Z^x_{s,t} )] = \int_{\R^d} \phi(y) N_{e^{(t-s)A} x + F_{s,t}, Q_{t-s}} (\d y) .$$
Under our assumptions, the matrix $Q_{t-s}$ is invertible. Therefore, denoting by $N_{Q_{t-s}}= N_{0,Q_{t-s}}$, the Gaussian distribution centered at the origin, we have
$$\aligned
\rho_{s,t}(x, y) := &\, \frac{\d N_{e^{(t-s)A} x + F_{s,t}, Q_{t-s}}}{\d N_{Q_{t-s}}} (y)\\
=&\, \exp\bigg\{ - \frac12 \big| Q_{t-s}^{-1/2} \big( e^{(t-s)A} x + F_{s,t} \big) \big|^2 + \big\< Q_{t-s}^{-1/2} \big(e^{(t-s)A} x + F_{s,t} \big), Q_{t-s}^{-1/2}y \big\> \bigg\}.
\endaligned $$
Using this notation we can write
$$S_{s,t}\phi(x) = \int_{\R^d} \phi(y)  \rho_{s,t}(x, y) N_{Q_{t-s}} (\d y). $$
As a result, for any $h\in \R^d$,
$$\aligned
&\, \frac{\d}{\d \eps}\Big|_{\eps=0} \big[ S_{s,t}\phi(x+\eps h) \big]  \\
=&\, \int_{\R^d} \phi(y) \rho_{s,t}(x, y) \Big[- \big\< Q_{t-s}^{-1/2} \big( e^{(t-s)A} x + F_{s,t} \big), \Lambda(t-s) h \big\> +  \big\< \Lambda(t-s) h, Q_{t-s}^{-1/2}y \big\> \Big]  N_{Q_{t-s}} (\d y) \\
=&\, \int_{\R^d} \phi(y) \rho_{s,t}(x, y) \big\< \Lambda(t-s) h, Q_{t-s}^{-1/2} \big(y - e^{(t-s)A} x - F_{s,t} \big) \big\> N_{Q_{t-s}} (\d y) .
\endaligned $$
Therefore,
$$\aligned
\big\<h, D(S_{s,t}\phi)(x) \big\> &= \int_{\R^d} \phi(y) \big\< \Lambda(t-s) h, Q_{t-s}^{-1/2} \big(y - e^{(t-s)A} x - F_{s,t} \big) \big\> N_{e^{(t-s)A} x + F_{s,t}, Q_{t-s}}(\d y) \\
&= \E\Big[ \phi(Z^x_{s,t}) \big\<\Lambda(t-s)h, Q_{t-s}^{-1/2} \big(Z^x_{s,t} - e^{(t-s)A} x - F_{s,t} \big) \big> \Big].
\endaligned $$
The proof is complete.
\end{proof}

\bigskip

\noindent \textbf{Acknowledgements.} The second author would like to thank the financial supports of the National Natural Science Foundation of China (Nos. 11688101, 11931004), and the Youth Innovation Promotion Association, CAS (2017003).

\end{document}